\documentclass[english,12pt]{amsart}
\usepackage{amssymb,amsfonts,amsmath}
\usepackage{amscd}
\usepackage[all]{xypic}
\usepackage{color}
\usepackage{enumerate}
\usepackage{todonotes}

\usepackage{thmtools} 

\CompileMatrices

\def\ord{{\rm ord}}

\def\Gal{{\rm Gal}}
\def\Aut{{\rm Aut}}

\def\Gr{{\rm Gr}}

\def\11{{\mathbf 1}}
\def\AA{{\mathbb A}}

\def\CC{{\mathbb C}}

\def\FF{{\mathbb F}}

\def\PP{{\mathbb P}}
\def\QQ{{\mathbb Q}}
\def\RR{{\mathbb R}}

\def\ZZ{{\mathbb Z}}

\def\cO{{\mathcal O}}

\mathchardef\alphag="7C0B \mathchardef\betag="7C0C
\mathchardef\gammag="7C0D \mathchardef\deltag="7C0E
\mathchardef\varepsilong="7C22 \mathchardef\varphig="7C27
\mathchardef\psig="7C20 \mathchardef\zetag="7C10
\mathchardef\epsilong="7C0F \mathchardef\rhog="7C1A
\mathchardef\taug="7C1C \mathchardef\upsilong="7C1D
\mathchardef\iotag="7C13 \mathchardef\thetag="7C12
\mathchardef\pig="7C19 \mathchardef\sigmag="7C1B
\mathchardef\etag="7C11 \mathchardef\omegag="7C21
\mathchardef\kappag="7C14 \mathchardef\lambdag="7C15
\mathchardef\mug="7C16 \mathchardef\xig="7C18
\mathchardef\chig="7C1F \mathchardef\nug="7C17
\mathchardef\varthetag="7C23 \mathchardef\varpig="7C24
\mathchardef\varrhog="7C25 \mathchardef\varsigmag="7C26
\mathchardef\Omegag="7C0A \mathchardef\Thetag="7C02
\mathchardef\Sigmag="7C06 \mathchardef\Deltag="7C01
\mathchardef\Phig="7C08 \mathchardef\Gammag="7C00
\mathchardef\Psig="7C09 \mathchardef\Lambdag="7C03
\mathchardef\Xig="7C04 \mathchardef\Pig="7C05
\mathchardef\Upsilong="7C07

\newcounter{theoremcntr}[subsection]
\renewcommand*{\thetheoremcntr}{%
  \ifnum\value{subsection}=0 %
    \thesection
  \else
    \thesubsection
  \fi
  .\arabic{theoremcntr}%
}
\numberwithin{equation}{subsection}
\renewcommand*{\theequation}{%
  \ifnum\value{subsection}=0 %
    \thesection
  \else
    \thesubsection
  \fi
  .\arabic{equation}%
}

\providecommand{\corollaryname}{Corollary}
\providecommand{\definitionname}{Definition}
\providecommand{\notationname}{Notation}
\providecommand{\lemmaname}{Lemma}
\providecommand{\propositionname}{Proposition}
\providecommand{\remarkname}{Remark}
\providecommand{\theoremname}{Theorem}
\providecommand{\conjecturename}{Conjecture}
\providecommand{\examplename}{Example}

\numberwithin{equation}{section}
\numberwithin{figure}{section}
\theoremstyle{plain}
\newtheorem{thm}{\theoremname}[section]
\theoremstyle{plain} 
\newtheorem{maintheorem}[thm]{\protect\theoremname}
\theoremstyle{remark}
\newtheorem{remark}[thm]{\remarkname}
\theoremstyle{definition}
\newtheorem{defn}[thm]{\definitionname}
\theoremstyle{plain}
\newtheorem{notation}[thm]{\notationname}
\theoremstyle{plain}
\newtheorem{cor}[thm]{\corollaryname}
\theoremstyle{plain}
\newtheorem{prop}[thm]{\propositionname}
\theoremstyle{plain}
\newtheorem{lem}[thm]{\lemmaname}
\theoremstyle{plain}
\newtheorem{conj}[thm]{\conjecturename}
\theoremstyle{plain}
\newtheorem{example}[thm]{\examplename}
\theoremstyle{plain}

\RequirePackage[colorlinks=true, breaklinks=true, urlcolor= black, linkcolor= black, citecolor= black, bookmarksopen=true,linktocpage=true,plainpages=false,pdfpagelabels,unicode]{hyperref}

\usepackage{cleveref}

\crefname{remark}{Remark}{Remarks}
\Crefname{remark}{Remark}{Remarks}

\crefname{defn}{Definition}{Definitions}
\Crefname{defn}{Definition}{Definitions}

\crefname{notation}{Notation}{Notations}
\Crefname{notation}{Notation}{Notations}

\crefname{cor}{Corollary}{Corollaries}
\Crefname{cor}{Corollary}{Corollaries}

\crefname{prop}{Proposition}{Propositions}
\Crefname{prop}{Proposition}{Propositions}

\crefname{lem}{Lemma}{Lemmas}
\Crefname{lem}{Lemma}{Lemmas}

\crefname{conj}{Conjecture}{Conjectures}
\Crefname{conj}{Conjecture}{Conjectures}

\crefname{example}{Example}{Examples}
\Crefname{example}{Example}{Examples}

\crefname{thm}{Theorem}{Theorems}
\Crefname{thm}{Theorem}{Theorems}

\crefname{maintheorem}{Theorem}{Theorems}
\Crefname{maintheorem}{Theorem}{Theorems}

\def\boxit#1#2{\setbox1=\hbox{\kern#1{#2}\kern#1}%
\dimen1=\ht1 \advance\dimen1 by #1 \dimen2=\dp1 \advance\dimen2 by
#1
\setbox1=\hbox{\vrule height\dimen1 depth\dimen2\box1\vrule}%
\setbox1=\vbox{\hrule\box1\hrule}%
\advance\dimen1 by .4pt \ht1=\dimen1 \advance\dimen2 by .4pt
\dp1=\dimen2 \box1\relax}

\mathchardef\alphag="7C0B \mathchardef\betag="7C0C
\mathchardef\gammag="7C0D \mathchardef\deltag="7C0E
\mathchardef\varepsilong="7C22 \mathchardef\varphig="7C27
\mathchardef\psig="7C20 \mathchardef\zetag="7C10
\mathchardef\epsilong="7C0F \mathchardef\rhog="7C1A
\mathchardef\taug="7C1C \mathchardef\upsilong="7C1D
\mathchardef\iotag="7C13 \mathchardef\thetag="7C12
\mathchardef\pig="7C19 \mathchardef\sigmag="7C1B
\mathchardef\etag="7C11 \mathchardef\omegag="7C21
\mathchardef\kappag="7C14 \mathchardef\lambdag="7C15
\mathchardef\mug="7C16 \mathchardef\xig="7C18
\mathchardef\chig="7C1F \mathchardef\nug="7C17
\mathchardef\varthetag="7C23 \mathchardef\varpig="7C24
\mathchardef\varrhog="7C25 \mathchardef\varsigmag="7C26
\mathchardef\Omegag="7C0A \mathchardef\Thetag="7C02
\mathchardef\Sigmag="7C06 \mathchardef\Deltag="7C01
\mathchardef\Phig="7C08 \mathchardef\Gammag="7C00
\mathchardef\Psig="7C09 \mathchardef\Lambdag="7C03
\mathchardef\Xig="7C04 \mathchardef\Pig="7C05
\mathchardef\Upsilong="7C07

\def\e{\varepsilon}

\DeclareMathOperator*{\Kr}{{Kr}}

\DeclareMathOperator{\charac}{char}

\newcommand{\sing}{\mathrm{sing}}
\newcommand{\ns}{\mathrm{ns}}

\newcommand{\sepc}{\operatorname{sep}}
\newcommand{\redu}{\operatorname{red}}

\def\ord{{\rm ord}}

\definecolor{orange}{rgb}{1,0.5,0}

\definecolor{blue}{rgb}{0,0,1}

\title[Dimension growth and Hilbert's irreducibility]
{Improvements on dimension growth results and effective Hilbert's irreducibility theorem}

\author[Cluckers]{Raf Cluckers}
\address{Univ.~Lille, CNRS, UMR 8524 - Laboratoire Paul Painlev\'e, F-59000 Lille, France, and
KU Leuven, Department of Mathematics, B-3001 Leu\-ven, Bel\-gium}
\email{Raf.Cluckers@univ-lille.fr}
\urladdr{http://rcluckers.perso.math.cnrs.fr/}

\author[Dèbes]{Pierre Dèbes}
\address{Univ.~Lille, UMR 8524 - Laboratoire Paul Painlev\'e, F-59000 Lille, France}
\email{pierre.debes@univ-lille.fr}
\urladdr{https://pro.univ-lille.fr/pierre-debes}

\author[Hendel]{Yotam I.~Hendel}
\address{Department of Mathematics, Ben-Gurion University of the Negev, P.O.B. 653, Be’er Sheva 84105, Israel}
\email{yhendel@bgu.ac.il}
\urladdr{https://sites.google.com/view/yotam-hendel}

\author[Nguyen]{Kien Huu Nguyen}
\address{KU Leuven, Department of Mathematics,
B-3001 Leu\-ven, Bel\-gium, and, 
Institute of Mathematics, Vietnam Academy of Science and Technology, 18 Hoang Quoc Viet, Nghia Do, Hanoi, Vietnam}
\email{kien.nguyenhuu@kuleuven.be,nhkien@math.ac.vn}
\urladdr{https://sites.google.com/site/nguyenkienmath/home}

\author[Vermeulen]{Floris Vermeulen}
\address{Mathematics M\"unster, University of M\"unster, Germany}
\email{florisvermeulen.math@gmail.com}
\urladdr{https://sites.google.com/view/floris-vermeulen}

\subjclass[2020]{Primary 11D45, 14G05, 12E25; Secondary 11G35, 11G50, 11R09, 11C08}
\keywords{Bertini's theorem, dimension growth conjecture, determinant method, heights in global fields, Hilbert's irreducibility theorem, number of rational solutions of Diophantine equations, rational points of bounded height, specialization, varieties over global fields}

\begin{document}

\begin{abstract}
We sharpen and generalize the dimension growth bounds for the number of points of bounded height lying on an irreducible algebraic variety of degree $d$, over any global field. In particular, we focus on the affine hypersurface situation by relaxing the condition on the top degree homogeneous part of the polynomial describing the affine hypersurface, while sharpening the dependence on the degree in the bounds compared to previous results. We formulate a conjecture about plane curves which provides a conjectural approach to the uniform degree $3$ case (the only remaining open case).
For induction on dimension, we develop a higher dimensional effective version of Hilbert's irreducibility theorem, which is of independent interest.
\end{abstract}

\maketitle

\section{Introduction and main results}\label{sec:intro}

This paper improves the state of the art of the uniform dimension growth conjecture in the affine situation, with some consequences in the projective situation as well.

We explain our affine results from Section \ref{subsec:aff} on, but we first sketch the simpler projective situation. Recall the uniform dimension growth conjecture from Conjecture 2 of \cite{Heath-Brown-Ann}, following questions raised in \cite[p.~227]{Heath-Brown-cubic},\,\cite[p.\,178]{Serre-Mordell}, and \cite[p.~27]{Serre-Galois}. For a projective variety $X\subset \PP^n$, write
$$
N(X,B)
$$
for the number of points $x$ on $X(\QQ)$ of height at most $B$, namely, that can be written in homogeneous coordinates $(x_0: \ldots: x_n)$ for some integers $x_i$ with $|x_i|\leq B$ for $i=0,\ldots,n$.

\begin{conj}[Uniform dimension growth conjecture \cite{Heath-Brown-Ann}]\label{con:dgcdegree:non-unif}
Fix $\varepsilon>0$ and integers $d \ge 2$ and $n \ge 3$. Then there is a constant $c_{d,n,\varepsilon}$ such that for any integral hypersurface $X\subseteq \PP_\QQ^n$ of degree $d$, one has for all $B\geq 1$
\begin{equation}\label{eq:dgc:dim X:non-unif}
N(X,B) \leq c_{d,n,\varepsilon}  B^{\dim X+\varepsilon}.
\end{equation}
\end{conj}

Heath-Brown proves the cases $d=2$ (for all $n$) and $n=3$ (for all $d \ge 2$) in Theorems 2 and 9 of \cite{Heath-Brown-Ann}. Furthermore, he introduces a $p$-adic variant of the determinant method, which has been strengthened and generalized in \cite{BrowningHeathBrown-Crelle}, \cite{Salb:d3}, \cite{Salberger-dgc}, \cite{Walsh}, \cite{CCDN-dgc}, \cite{Vermeulen:p}, \cite{Pared-Sas}, \cite{BCN-d}. In particular, Conjecture~\ref{con:dgcdegree:non-unif} is now known when $d\ge 4$ while removing the $\varepsilon$ and making the dependence on $d$ polynomial.
We improve these results further by showing the following.
\begin{thm}[Uniform dimension growth for projective hypersurfaces]\label{thm:dgcdegree:proj}
Fix $n \ge 3$. Then there is a constant $c=c(n)$ such that for any integral hypersurface $X\subseteq \PP_\QQ^n$ of degree $d\ge 5$, one has for all $B\geq 1$
\begin{equation}\label{eq:thm:dgc:dim X:non-unif}
N(X,B) \leq c d^7  B^{\dim X}.  
\end{equation}
\end{thm}
The new aspect of \Cref{thm:dgcdegree:proj} is the factor $d^7$, improving the factor $d^{e(n)}$ for some $e(n)$ depending on $n$ from Theorem 1 of \cite{CCDN-dgc}.
Apart from the $p$-adic determinant method from \cite{Heath-Brown-Ann}, \cite{Salberger-dgc}, our methods involve new ingredients: a combination of Bertini's Theorem and an effective Hilbert's irreducibility theorem in higher dimension (see \Cref{sec:hilb,sec:eff.hilbert}), and, counting separately on linear subspaces of maximal dimension (see \Cref{sec:linear:sub}). We also obtain a variant of \Cref{thm:dgcdegree:proj} for any global field $K$ instead of $\QQ$, including positive characteristic, see \Cref{thm:dgcdegree:proj:K}.

By the work \cite{BrowningHeathBrown-Crelle}, one shows results like \Cref{thm:dgcdegree:proj} by showing underlying  affine counting results, namely, for affine hypersurfaces.  In this paper most of the work and focus is on this underlying affine situation, and, we present several improvements and generalizations for this affine situation, more than needed for obtaining \Cref{thm:dgcdegree:proj}. Let us now explain our affine results.




\subsection{Improvements to affine dimension growth}\label{sec:1.1}
\label{subsec:aff}
If $X\subset \AA^n_{\QQ}$ is an affine variety, define
\[
N_{\rm aff}(X, B)
\]
as the number of points $x = (x_1, \ldots, x_n)$ lying on $X(\ZZ)$ and satisfying $|x_i|\leq B$ for $i=1, \ldots, n$. Also write $N_{\rm aff}(f, B)$ for $N_{\rm aff}(X, B)$ when $X$ is a hypersurface given by a polynomial $f$.




In previous works on affine dimension growth for a degree $d\geq 2$ hypersurface defined by $f=0$, it was  typically assumed that the degree $d$ part $f_d$ of $f$ is absolutely irreducible, see \cite{Brow-Heath-Salb,Salb:d3,Salberger-dgc,CCDN-dgc,Pared-Sas}. We relax the conditions on $f_d$, see \Cref{thm:fd:Q,thm:fd:Q:conj} below.


Our first variant of affine dimension growth uses a condition on $f$ which we coin NCC, for `not cylindrical over a curve'. This 
condition means that $f$ depends non-trivially on at least three variables in any  affine coordinate system over our base field, which is $\QQ$ up to \Cref{sec:global}.

\begin{defn}\label{defn:H-cond}
Let $f$ be a polynomial over $\QQ$ in $n\ge 3$ variables and $X=V(f)$ the affine hypersurface cut by $f$. Say that $f$ is \emph{cylindrical over a curve} if there exists a $\QQ$-linear map $\ell : \AA^n_\QQ\to \AA^2_\QQ$ and a curve $C$ in $\AA^2_\QQ$ such that $X=\ell^{-1}(C)$. We abbreviate \emph{not cylindrical over a curve} by NCC and say that $X$ is NCC if $f$ is.
\end{defn}

Note that $f$ is cylindrical over a curve if and only if there exist a polynomial $g\in \QQ[y_1, y_2]$ and linear forms $\ell(x), \ell'(x)$ over $\QQ$ such that $f(x) = g(\ell(x), \ell'(x))$.

\begin{defn}\label{defn:m-irreducible}
Let $f$ be a polynomial over $\QQ$ and let $r$ be a positive integer. We say that $f$ is \emph{$r$-irreducible} if $f$ does not have any factors of degree $\leq r$ over $\QQ$.
\end{defn}

The following is our main result on affine dimension growth over $\QQ$, for NCC hypersurfaces.
\begin{maintheorem}[Affine dimension growth]\label{thm:fd:Q}
Given $n\geq 3$ an integer, there exist constants $c = c(n), e = e(n)$ such that for all  polynomials $f$ in $\ZZ[y_1, \ldots, y_n]$ of degree $d\geq 3$ so that $f$ is irreducible over $\QQ$ and NCC, and whose homogeneous degree $d$ part $f_d$ is $2$-irreducible, one has for all $B\geq 2$ that
\begin{align*}
N_{\rm aff}(f, B)&\leq cd^7 B^{n-2}, &\text{ if } d\geq 5, \\
N_{\rm aff}(f, B)&\leq c B^{n-2} (\log B)^e, &\text{ if } d = 4, \\
N_{\rm aff}(f, B)&\leq c B^{n-3 + 2/\sqrt{3}} (\log B)^e, &\text{ if } d = 3.
\end{align*}
If $f_d$ is only $1$-irreducible (instead of $2$-irreducible), then we still have that
\begin{align*}
N_{\rm aff}(f, B)&\leq cd^7 B^{n-2} (\log B)^e, &\text{ if } d \geq 4, \\
N_{\rm aff}(f, B)&\leq c B^{n-3 + 2/\sqrt{3}} (\log B)^e, &\text{ if } d = 3.
\end{align*}
\end{maintheorem}
\begin{remark}
If $f$ is not NCC, then determining $N_{\rm aff}(f, B)$ reduces to counting integral points on a certain curve. Indeed, if $f$ is not NCC, then there exists a $\QQ$-linear map $\ell: \AA^n_{\QQ}\to \AA^2_\QQ$ and a curve $C\subset \AA^2_\QQ$ such that $V(f) = \ell^{-1}(C)$. Hence
\[
N_{\rm aff}(f, B) \leq c B^{n-2} N_{\rm aff}(C, B),
\]
where the constant $c$ depends only on $n$, see e.g.~\cite[Lemma 4.5]{Browning-Q}. In particular, one cannot hope to remove the NCC condition in this generality, as the example $x_1 = x_2^d$ considered as a hypersurface in $\AA^n$ shows.
\end{remark}
By the proof of \cite[Proposition 5]{CCDN-dgc}, the exponent of $d$ cannot drop below $2$ in the bound for $d\ge 5$ in Theorem \ref{thm:fd:Q}. In \cite[Theorem 1.2]{Vermeulen:affinedg}, Vermeulen removes the condition on $f_d$ from Theorem \ref{thm:fd:Q}, at the cost of losing the polynomial dependence on $d$ and reintroducing a factor of $B^\e$.
In Section~\ref{sec:global}, we give a variant of \Cref{thm:fd:Q} for any global field $K$ instead of $\QQ$ in \Cref{thm:fd:K}, which generalizes work by Vermeulen~\cite{Vermeulen:p} and Paredes--Sasyk~\cite{Pared-Sas}.


\subsection{A conjectural approach}\label{subsec:conj}
Our second affine variant is given by \Cref{thm:fd:Q:conj} and is based on Conjecture  \ref{conj:uniform.curve.conjecture}.
We first give some context to motivate this conjecture.

For curves, recall the following variant of a classical finiteness result, which follows from Siegel's theorem as stated in e.g.\ ~\cite[Thm.\,D.8.4, Thm.\,D.9.1, Rem.\,D.9.2.2]{hindrysilverman}, see also \cite{Schinzel}.
\begin{thm}[Siegel
]\label{thm: Siegel}
Given $f$ in $\ZZ[x,y]$ whose homogeneous part of highest degree 
is $2$-irreducible, then there are only finitely many solutions in $\ZZ^2$ to the equation $f(x,y)=0$.
\end{thm}

Also recall a conjecture by W. M. Schmidt, just below (1.4) in \cite{Schmidt-points}, which conjectures a bound $c(d, \varepsilon)B^\varepsilon$ for integer points up to height $B$ on a planar curve of degree $d$ and positive genus, with surprising evidence in \cite{Dim-Gao-Hab}.

Motivated by the finiteness result of \Cref{thm: Siegel} for $d\ge 3$, by Schmidt's conjecture for positive genus, and by Theorem 2 of \cite{bombieri-pila}, we introduce a question which includes all genera. 

\begin{conj}\label{conj:uniform.curve.conjecture}
Fix $d\ge 2$ and $\varepsilon>0$. Then there exists a constant $c_{d,\varepsilon}$ such that for all $f$ in $\ZZ[x,y]$ of degree $d$ whose homogeneous part of highest degree $f_d$ is $1$-irreducible
and all $B\geq 1$ one has
\begin{equation}\label{con1:eq}
\tag{1.3}
N_{\rm aff}(f,B) \leq c_{d,\varepsilon}  B^{\varepsilon}.
\end{equation}
\end{conj}
Conjecture \ref{conj:uniform.curve.conjecture} as above 
is in line with a more abstract conjecture on $\ell$-primary parts of class groups of number fields, as given in  \cite[Conjecture 3.5]{Zhang-equi}. Similarly, bounds on the number of integral points lying on planar curves also play an important role in the recent work  \cite{BSTTTZ20}, where the authors bound the  number of $2$-torsion elements in the class group of a number field of degree at least $3$ (see \cite[Section 1.6]{CCDN-dgc} for a longer discussion). 

The case of Pell or Thue equations forms some evidence for Conjecture \ref{conj:uniform.curve.conjecture}. In general, we expect this conjecture to be very difficult. Note that if $f$ is irreducible (and no condition on $f_d$), then by the well-known result by Bombieri and Pila~\cite[Theorem 5]{bombieri-pila}, the above conjecture holds as soon as $\varepsilon>1/d$.
 

We obtain affine dimension growth results for those values $d\geq 3$ and $\varepsilon>0$ for which
Conjecture  \ref{conj:uniform.curve.conjecture} holds, see \Cref{thm:fd:Q:conj}. This presents a (conjectural) approach to the uniform dimension growth conjecture (Conjecture \ref{con:dgcdegree:non-unif}) when $d=3$, since the projective case follows from the affine case. The proof of \Cref{thm:fd:Q:conj} uses a basic induction argument with base case given by Conjecture  \ref{conj:uniform.curve.conjecture} and induction step enabled by a higher dimensional effective form of Hilbert's irreducibility theorem. As far as we can see, the dimension growth conjecture for degree $3$ projective varieties is at best partially known: either non-uniformly (with $c$ depending furthermore on $X$ in Conjecture \ref{con:dgcdegree:non-unif}) in \cite{Salberger-dgc}, or, assuming $\varepsilon>1/7$ in \cite{Salb:d3}, or, with $n=3$ in \cite{Heath-Brown-Ann}. 

\begin{maintheorem}[Affine dimension growth, assuming Conjecture \ref{conj:uniform.curve.conjecture}]\label{thm:fd:Q:conj}
Fix $d\geq 2$ and $\varepsilon>0$ and suppose that Conjecture \ref{conj:uniform.curve.conjecture} holds for this $d$ and $\varepsilon$. Let $n\ge 2$ be given. Then there exists a constant $c_{n,d,\varepsilon}$ such that for all degree $d$ polynomials $f$ in $\ZZ[y_1,\ldots,y_n]$
whose homogeneous part of highest degree $f_d$ is $1$-irreducible, and all $B\geq 1$ one has
$$
N_{\rm aff}(f,B) \leq c_{n,d,\varepsilon}  B^{n-2+\varepsilon}.
$$
\end{maintheorem}

\Cref{thm:fd:Q:conj} includes the planar case $n=2$, which coincides with Conjecture \ref{conj:uniform.curve.conjecture}. A sharpening of Conjecture \ref{conj:uniform.curve.conjecture} with an upper bound with $(\log B)^e$ for some $e=e(d)$ instead of the factor $B^\varepsilon$ would naturally lead to a corresponding improvement in \Cref{thm:fd:Q:conj}.



The proof of \Cref{thm:fd:Q:conj} follows a  simple induction argument by cutting with well-chosen hyperplanes, reducing it to  the planar case of Conjecture \ref{conj:uniform.curve.conjecture}. The choice of hyperplanes is enabled by a higher-dimensional effective form of Hilbert's irreducibility theorem given below as \Cref{thm:eff.Hil.lemma:Q}, which is a new ingredient compared to the induction argument from \cite{Brow-Heath-Salb}.

\subsection{Higher dimensional effective  variants of Hilbert's irreducibility theorem}\label{sec:hilb}

Given a polynomial $f$ over $\QQ$, let $\lVert f\rVert$ denote the height of its tuple of coefficients in projective space. Theorems \ref{thm:eff.Hil.lemma:Q} and \ref{thm:HIT.v2:Q} are our main results on higher dimensional effective Hilbert's irreducibility theorem.

\begin{maintheorem}\label{thm:eff.Hil.lemma:Q}
Let $r,n,d$ be positive integers,
 let $\underline{T}=(T_1, \ldots, T_r), \underline{Y}=(Y_1, \ldots, Y_n)$ be tuples of variables, and let
$F\in \ZZ[\underline{T},\underline{Y}]$ be an irreducible polynomial of degree $d$ in  $\QQ[\underline{T},\underline{Y}]$ such that $\deg_{\underline{Y}} F \geq 1$. For an integer $B \geq 1$ denote by $S_T(F, B)$ the number of $\underline{t}=(t_1,\ldots,t_r)\in [-B,B]^r$ for which $F(\underline{t},\underline{Y})$ is reducible in $\QQ[\underline{Y}]$. Then
\[
S_T(F, B) \ll_{n,r} 2^{c(r)(d+1)^n}  (\log \|F\| + 1)^{10} B^{r- 1/2}\log(B)^{10(r-1)},
\]
for every $B \ge 2$, where $c(r)=118+10(r-1)$.
\end{maintheorem}


As a corollary, we obtain an effective high dimensional version of Hilbert's irreducibility theorem for several polynomials.
Since for applications one might need a good tuple $t$ which avoids a given  hypersurface, we state it accordingly.
See ~\cite[Th\'{e}orèmes 2, 3]{Walkowiak} for a similar formulation in the one dimensional case.

\begin{cor}\label{cor:eff:Hil}
Let $r,n,m,d$ be positive integers,
let $\underline{T}=(T_1, \ldots, T_r), \underline{Y}=(Y_1, \ldots, Y_n)$ be tuples of variables, and let $F_1, \ldots, F_m$ be irreducible polynomials in $\QQ[\underline{T}, \underline{Y}]$ of degree at most $d$ such that $\deg_{\underline{Y}}(F_i) \ge 1$ for each $i$.
Then there exists a polynomial $G$ of degree at most $30$ in $m+1$ variables whose coefficients depend on $r,n,m,d$
 such that the following holds for every $s \ge 1$. If $X \subset \AA^r_\QQ$ is a (possibly reducible) hypersurface  of degree  at most $s$, then we have a
 tuple of integers $\underline{t} \in \ZZ_{\ge 0}^{r}$ not contained in $X(\QQ)$ of height \[
 \ll_{n,r} G(\log \lVert F_1\rVert, \ldots, \log \lVert F_m\rVert,s),
\]
 such that for each $1 \le i \le m$
 the polynomial $F_i(\underline{t},\underline{Y})$ is irreducible over $\QQ$.
\end{cor}


The planar case  of \Cref{thm:eff.Hil.lemma:Q} where $r=n=1$ is covered by \cite{Walkowiak}, \cite{DebesW}, \cite{Schinzel-Zann}, and \cite{Pared-Sas:Hilbert}.  
We first prove \Cref{thm:eff.Hil.lemma:Q} in the case where $r=1$ and $n > 1$. It is based on the case $r=n=1$ from \cite{Walkowiak} and  \cite{Pared-Sas:Hilbert}, and on the use of Kronecker transforms.
We then use induction to prove the $r>1$ case.

In fact, for our purposes we only need the case where $r=1$ and $n\ge 2$, and  so the last induction argument is not strictly necessary for our dimension growth results. \Cref{thm:eff.Hil.lemma:Q} and \Cref{cor:eff:Hil} as given above correspond to the general situation of Hilbert's irreducibility theorem, and are of independent interest.

Finally, we also obtain the following result resembling both Bertini and effective Hilbert's irreducibility theorem, which does not include dependence on the height of $F$, and which we use in the proof of \Cref{thm:fd:Q} to take care of the case where the degree of $f$ is small. Let $(\PP^{n-1}_{\QQ})^*$ denote the dual space of $\PP^{n-1}_{\QQ}$, which parametrizes hyperplanes in $\AA^n_{\QQ}$ through the origin. We will identify elements of $(\PP^{n-1}_{\QQ})^*$ with linear forms  in $n$ variables (up to a scalar).

\begin{thm}
\label{thm:HIT.v2:Q}
Let $n \ge 4$, and let $F \in \ZZ[X_1,\ldots,X_n]$ be an $r$-irreducible homogeneous polynomial for an integer $r \ge 1$. Then there exists a hypersurface $W \subset (\mathbb{P}^{n-1}_{\QQ})^*$ of degree at most $O_n(d^{3})$ such that if $\ell \notin W$ is a $\QQ$-linear form, then $F|_{\{\ell=0\}}$ is $r$-irreducible over $\QQ$.
\end{thm}

The proof of \Cref{thm:HIT.v2:Q} uses Bertini's irreducibility theorem in an essential way, and therefore works for homogeneous polynomials in at least $4$ variables.


In \Cref{sec:global}, we naturally obtain versions of our main results for all global fields, namely  Theorems \ref{thm:fd:Q} up to
\ref{thm:HIT.v2:Q} which generalize and improve recent results for global fields from \cite{Vermeulen:p}, \cite{Pared-Sas} and  \cite{Pared-Sas:Hilbert}.

\subsection*{Acknowledgements} The authors would like to thank Rami Aizenbud, Gal Binyamini, Tim Browning, Wouter Castryck, Vesselin Dimitrov, Philip Dittmann, Ziyang Gao, Itay Glazer, Marcelo Paredes, Roman Sasyk and Stephan Snegirov for interesting discussions on the topics of the paper.
We would like to particularly thank Per Salberger for sharing a draft of his work \cite{Salb.upcoming}, and for an email correspondence which inspired us to prove a further variant of Theorem \ref{thm:fd:Q},
where the conditions of $2$-irreducibility and $1$-irreducibility of $f_d$ are replaced by alternative ones (see \Cref{thm:final}). We also thank the anonymous referee for their valuable time in reviewing the manuscript. 

R.C., P.D.~and Y.I.H.~were partially supported by the Labex CEMPI (ANR-11-LABX-0007-01).
R.C.~was partially supported by KU Leuven IF C16/23/010.
Y.I.H.~was partially supported by FWO Flanders (Belgium) with grant number 12B4X24N. K.H.N.~was partially supported by FWO Flanders (Belgium) with grant number 1270923N and by the Excellence Research Chair “FLCarPA: L-functions in positive characteristic and applications” financed by the Normandy Region. F.V.~was supported by FWO Flanders (Belgium) with grant number 11F1921N. R.C., P.D.~and F.V.~would like to thank the CNRS France-Japan AHGT International Research Network.

\subsection*{Notation} For functions $f, g$, and some data $n=n_1,n_2,\ldots, n_\ell$, we write that $f = O_n(g)$, or $f\ll_n g$ if there exists a constant $c = c(n)$ depending only on $n$ such that $|f| \leq c |g|$. We write $f = o(g)$ if $\lim\limits_{x \to \infty} (f/g)(x) = 0$.

For a polynomial $f$  over a field $K$, we denote by $V(f)$ the corresponding variety defined by $f = 0$. Depending on context, this is either a projective or affine variety. If $f$ is of degree $d$, we denote by $f_d$ its  homogeneous degree $d$ part.

\section{Preliminaries}\label{sec:prelim}

To prove our results, we need to construct certain auxiliary polynomials, for which we use~\cite[Prop.\,4.2.1]{CCDN-dgc}.

\begin{prop}[{\cite[Proposition\,4.2.1]{CCDN-dgc}}]\label{prop:aux.affine.poly}
Let $n\geq 1$ be an integer and let $f\in \ZZ[x_1, \ldots, x_{n+1}]$ be primitive and irreducible of degree $d\geq 1$. Then for each integer $B\geq 1$ there exists a polynomial $g\in \ZZ[x_1, \ldots, x_{n+1}]$ of degree at most
\[
\ll_n d^{3-1/n} B^{\frac{1}{d^{1/n}}} (\log B+d),
\]
%
not divisible by $f$ and vanishing on all points of $x\in \ZZ^{n+1}$ with $f(x) = 0$ and $|x_i|\leq B$.
\end{prop}

\begin{proof}
Write $f_d$ for the degree $d$ homogeneous part of $f$.
Using ~\cite[Proposition\,4.2.1]{CCDN-dgc} with the first term of the minimum, we obtain such an interpolating polynomial $g$ of degree at most
\[
\ll_n B^{1/d^{1/n}} d^{2-1/n} \frac{\log \lVert f_d\rVert + d\log B + d^2}{\lVert f_d\rVert^{\frac{1}{nd^{1+1/n}}}} + d^{1-1/n}\log B + d^{4-1/n}.
\]
Now, we use the fact that the function
\[
 h(x)=\frac{\log(x)}{ x^{\frac{1}{nd^{1+1/n}}}}
 \]
  is bounded by $nd^{1+1/n}/e$ on the interval $[1, \infty)$  to get the desired result (we amend \cite[Lemma 4.2.2]{CCDN-dgc} which is used in the proof of \cite[Proposition\,4.2.1]{CCDN-dgc}, that the dependence on $d$ should have been made explicit and that $g$ should be coprime with $f$; this is done in both \cite[Lemma 5.12]{Pared-Sas} and in \cite[Lemma 3.5]{Vermeulen:p} and is harmless for the proof of \cite[Proposition\,4.2.1]{CCDN-dgc}).
\end{proof}

Combining this with a projection argument as for \cite[Proposition\,4.3.1]{CCDN-dgc} gives the following.
Note that the proofs of \cite[Propositions\,4.3.1 and 4.3.2]{CCDN-dgc} refer to \cite{Harris} which assumes characteristic zero, but the result still goes through in the positive characteristic case, by the basic theory of Chow varieties, see e.g.~\cite[Prop.\,8.3]{Rydh}, \cite[Ex.\,4.9\,Part I]{hartshorne}. We amend Theorem 5.9 of \cite{Pared-Sas}, where the degree of $\varphi(C)$ is not sufficiently controlled in the proof, and, an extra factor $d$ appears which is forgotten in the upper bound of the statement of the theorem; to avoid these problems, one should reason as for \cite[Proposition\,4.3.1]{CCDN-dgc}.

\begin{cor}[{{\cite[Theorem\,3]{CCDN-dgc}}}]\label{cor:counting.affine.curve}
Let $n\geq 2$ be an integer and let $C\subset \AA^n_{\QQ}$ be an affine irreducible curve of degree $d$. Then for every integer $B \ge 1$,
\[
N_{\rm aff}(C, B) \ll_n d^3 B^{1/d}(\log B + d).
\]
\end{cor}

In \cite{BCN-d}, the authors show that the exponent of $d$ can be improved (in both the affine and projective cases) at the cost of additional logarithmic factors.

\begin{thm}[{{\cite[Theorem 2]{BCN-d}}}]\label{thm:BCN-d}
Let $n\geq 2$ be an integer,
 let $C\subset \AA^n_{\QQ}$ be an irreducible  curve of degree $d$ and let $\overline{C}\subset \PP^n_{\QQ}$ denote its closure in projective space. Then for every integer $B \ge 2$,
\begin{flalign*}
N_{\mathrm{aff}}(C, B)\ll_n d^2B^{1/d} \log(B)^{\kappa}, \\
N(C, B)\ll_n d^2B^{2/d} \log(B)^{\kappa},
\end{flalign*}
where $\kappa$ is a universal constant.
\end{thm}

The quadratic dependence on $d$ in \Cref{thm:BCN-d} is known to be optimal (up to the $\log(B)^{\kappa}$ factor), by Section 6 of \cite{CCDN-dgc}. Recall also the Schwartz--Zippel bound, sometimes also called the trivial bound.

\begin{prop}[Schwartz-Zippel]\label{prop:schwarz--zippel}
Let $X\subset \AA^n_{\overline{\QQ}}$ be a variety of pure dimension $m$ and degree $d$ defined over $\overline{\QQ}$. Then
\[
N_{\rm aff}(X,B)\leq d(2B+1)^m.
\]
\end{prop}

\begin{proof}
See e.g.\ ~\cite[Sec.\,2]{Heath-Brown-Ann} or~\cite[Thm.\,1]{Browning-HB}.
\end{proof}

We will use effective Noether forms throughout the paper. We recall these below.
\begin{thm}[{{{\cite[Satz.\,4]{RuppertCrelle}, \cite[Thm.\,7]{KALTOFEN1995}}}}]\label{thm:noether.forms}
Let $d\geq 2, n\geq 3$ and let $K$ be a field. Put $C_K = 2$ if $\charac K = 0$ and $C_K = 6$ if $\charac K > 0$. Then there is a collection of polynomials $\{F_\ell\}_\ell$ over $\ZZ$ in $\binom{n+d}{n}$ variables of degree $O(d^{C_K})$ such that for any homogeneous polynomial $f$ over $K$ of degree $d$ in $n+1$ variables:
\begin{itemize}
\item If $f$ is not absolutely irreducible, then each $F_\ell$ vanishes when applied to the coefficients of $f$.
\item If $f$ is absolutely irreducible, then there exists $F_\ell$ which does not vanish when applied to the coefficients of $f$.
\end{itemize}
\end{thm}

We need an effective version of Pila's result on integral points on affine varieties~\cite[Thm.\,A]{Pila-ast-1995} as below. The effective aspect here is the polynomial dependence on $d$.
\begin{prop}\label{prop:pila.hypersurface}
Let $X\subset \AA^n_{\QQ}$ be an irreducible hypersurface of degree $d$. Then
\[
N_{\rm aff}(X,B) \ll_n d^{3(n-1)} B^{n-2 + 1/d} ( \log B + d).
\]
\end{prop}

\begin{proof}
We prove this result by induction, where the base case is~\cite[Thm.\,3]{CCDN-dgc} for curves (i.e.\ when $n=2$).

Let $X = V(f)$ where $f\in \ZZ[x_1, \ldots, x_n]$ is irreducible of degree $d$, with $n\geq 3$. If $f$ is not absolutely irreducible, then we can immediately conclude by~\cite[Cor.\,4.1.4]{CCDN-dgc}. So assume that $f$ is absolutely irreducible and consider the polynomial
\[
f_{a,k}(x) = f(-\frac{a_2}{a_1}x_2 - \frac{a_3}{a_1}x_3 - \ldots - \frac{a_n}{a_1}x_n - k, x_2, \ldots, x_n)\in \QQ[x_2, \ldots, x_n],
\]
where $a_i, k\in \ZZ$, and $a_1 \neq 0$. By Bertini's theorem, the polynomial $f_{a,0}$ is still absolutely irreducible for generic choices of $a_i$. By using Noether forms as in~\cite[Lem.\,4.3.7]{CCDN-dgc}, see \Cref{thm:noether.forms}, we can find $a_1, \ldots, a_n$ with $|a_i|\leq d^2(d-1)$ for which $f_{a,0}$ is absolutely irreducible. Let $F$ be a Noether form as in \Cref{thm:noether.forms}
 witnessing that $f_{a,0}$ is absolutely irreducible, and consider $F$ as a polynomial in $k$ by plugging in the coefficients from $f_{a,k}$. This is a non-zero polynomial of degree at most $d^2(d-1)$ in $k$, and so setting $\ell_a(x)=\sum_i a_i x_i$, there are at most $d^2(d-1)$ values of $k$ for which the variety $
X\cap \{ \ell_a(x)=k\}
$
is reducible over $\overline{\QQ}$.
Using \Cref{prop:schwarz--zippel} if the intersection $X\cap \{ \ell_a(x)=k\}$ is reducible, and the induction assumption otherwise, we get
\begin{align*}
N_{\rm aff}(X,B)&\leq
\sum_{|k|\leq d^2(d-1)B} N_{\rm aff}\left( X\cap \{ \ell_a(x)=k\}, B\right) \\
 &\ll_n  d^3(d-1) B^{n-2} + d^2(d-1)B d^{3(n-2)}B^{n-3+1/d}(\log B + d).
\end{align*}
This proves the result.
\end{proof}

A projection argument allows us to establish the following.

\begin{prop}\label{prop:pila}
Let $X\subset \AA^n_{\QQ}$ be an irreducible variety of dimension $m$ and degree $d$. Then
\[
N_{\rm aff}(X, B) \ll_n d^{3m + 2(n-m-1)^2(m-1)} B^{m-1 + 1/d} (\log B + d).
\]
\end{prop}

\begin{proof}
We apply~\cite[Prop.\,4.3.1]{CCDN-dgc} to obtain a hypersurface $X'\subset \AA^{m+1}_\QQ$ of degree $d$ birational to $X$ and satisfying
\[
N_{\rm aff}(X, B)\leq d N_{\rm aff}(X', O_n(1) d^{2(n-m-1)^2} B).
\]
The result then follows from the previous proposition.
\end{proof}

\section{Affine dimension growth}\label{sec:proof}

The goal of this section is to prove \Cref{thm:fd:Q,thm:fd:Q:conj}.

\subsection{Base of the induction for \Cref{thm:fd:Q}}\label{sec:base-case}

The following is the base for the induction argument for showing the affine dimension growth result of \Cref{thm:fd:Q}; it provides a more precise variant of \Cref{thm:fd:Q} when $n=3$.

\begin{prop}\label{prop:0.4strong}
There exist constants $c$ and $\kappa$ such that for all NCC polynomials $f$ in $\ZZ[y_1,y_2,y_3]$ of degree $d\geq 3$ which are irreducible over $\QQ$, and whose homogeneous degree $d$ part $f_d$ is $2$-irreducible, and for every integer $B \ge 2$, one has
$$
N_{\rm aff}(f,B)\leq c d^{9/2} B, \quad \mbox{ when $d\ge 5$},
$$
$$
N_{\rm aff}(f,B)\leq c (\log B)^\kappa B, \quad  \mbox{ when $d=4$ },
$$
and,
$$
N_{\rm aff}(f,B)\leq c (\log B)^\kappa  B^{2/\sqrt{3}} \quad  \mbox{ when $d=3$}.
$$
If $f_d$ is only  $1$-irreducible (instead of $2$-irreducible), we still have that
\begin{align*}
N_{\rm aff}(f, B) \leq c d^{9/2} B (\log B)^\kappa, \quad \mbox{ when $d\geq 4$}, \\
N_{\rm aff}(f,B)\leq c (\log B)^\kappa  B^{2/\sqrt{3}} \quad  \mbox{ when $d=3$}.
\end{align*}
\end{prop}

\Cref{prop:0.4strong} improves Proposition 4.3.4 of~\cite{CCDN-dgc} with a far lower exponent for $d$, and, with a factor $\log B$ instead of $B^\varepsilon$ when $d=3,4$. Precisely, the bound with a factor $d^{18}$ coming from the proof of Proposition~4.3.4 of \cite{CCDN-dgc} now becomes the factor $d^{9/2}$ (we hereby amend the statement of Proposition~4.3.4 of \cite{CCDN-dgc} that the $d^{14}$ should be $d^{18}$, coming from its proof; this correction is now obsolete given the strengthening of \Cref{prop:0.4strong}).

To prove~\Cref{prop:0.4strong}, we need to count integral points on lines contained in our surface. In fact, the integral points on lines will typically be the main contribution to counting all integral points. Counting on lines is done in \Cref{prop1:Brow-Heath-Salb} below, which improves Proposition 4.3.3 of \cite{CCDN-dgc} and Proposition 1 of \cite{Brow-Heath-Salb} (we amend Proposition 4.3.3 of \cite{CCDN-dgc}, Proposition 7.1 of \cite{Pared-Sas} and Theorem 0.4 of \cite{Salberger-dgc}
whose proofs use the condition that $f_d$ is absolutely irreducible, and thus this condition should be mentioned in their statements).
Note that we need to impose the NCC condition, otherwise our surface contains infinitely many parallel lines.

\begin{prop}\label{prop1:Brow-Heath-Salb}
There exists a constant $c$ such that the following claim holds for every NCC polynomial    $f\in \ZZ[x_1,x_2,x_3]$ of degree $d\geq 3$ whose highest degree part $f_d$ is $2$-irreducible.
Let $I$ be any finite set of affine lines lying on the hypersurface $X=V(f)$ and let $B\geq 1$ be an integer, then
\begin{equation}\label{bound:D}
N_{\rm aff}(X  \cap ( \cup_{L \in I}  L   )  ,B) \leq c d^4 B +(\# I).
\end{equation}

If $f_d$ is only $1$-irreducible (instead of $2$-irreducible), then for every integer $B \ge 2$,
\begin{equation}\label{bound:D2}
N_{\rm aff}(X  \cap ( \cup_{L \in I}  L   )  ,B) \leq c d^4 B \log B  + (\# I).
\end{equation}

\end{prop}

\begin{proof}[Proof of \Cref{prop1:Brow-Heath-Salb}]
Write $I=I_{1}\cup I_2$ where $I_1=\{L\in I\mid  N_{\rm aff}(L,B)\leq 1 
\}$ and $I_2=\{L\in I\mid N_{\rm aff}(L,B)>1\}$.
It is clear that $N_{\rm aff}(X\cap\cup_{L\in I_1}L)\leq \#I_1$. If $L\in I_2$, then there exist $a=(a_1,a_2,a_3), v=(v_1,v_2,v_3)\in \ZZ^{3}$ such that $H(a):=\max_i |a_i|\leq B$, $v$ is primitive and $L(\QQ)=\{a+t v\mid t\in\QQ\}$.
Since $v$ is primitive we deduce that
$$
L(\ZZ)\cap [-B,B]^{3}=\{a+t v\mid t\in \ZZ,\  H(a+t v)\leq B\}
$$
and thus since $L\in I_2$ it follows that $H(v)\leq 2B$, so we have
\[
N_{\rm aff}(L  ,B)=\#(L(\ZZ)\cap [-B,B]^{3})
\leq 1+\dfrac{2B}{H(v)}
\le \dfrac{4B}{H(v)}.
\]
We also have $f_{|L} = 0$, therefore when we write
 $f(a+tv)=\sum\limits_{i=0}^d c_i(a,v) t^i$, it follows that
$c_i(a,v)=0$ for every $0\le i \le d$. Explicitly,
\[
f(a+t v) = t^d f_d(v)
+ t^{d-1}\underbrace{\left(\sum a_i \frac{df_d}{dx_i}(v)+f_{d-1}(v)\right)}_{:=c_{d-1}(a,v)}
+\ldots + t \underbrace{\sum v_i \frac{df}{dx_i}(a)}_{:=c_{1}(a,v)}+ f(a)\equiv 0.
\]
Let $C$ denote the curve in  $\mathbb{P}^2_{\QQ}$ cut out by $f_d$ and fix $v_0 \in C(\QQ)$. We count the number of lines in $I$ in the direction of $v_0$. Because of the NCC condition, $c_1(a,v_0)$ is a non-zero polynomial in $a$. Indeed, otherwise there exists a linear change of variables such that $f$ only depends  on $2$ variables. Hence, there are at most $d(d-1)$ lines on $X$ in the direction of $v_0$, since they must all lie on the intersection of $X$ with the surface cut out by $c_1(a,v_0) = 0$, which is a $1$-dimensional variety of degree $d(d-1)$.
Factor $f_d = \prod_\ell g_\ell$ over $\QQ$, where $g_\ell$ is irreducible of degree $d_\ell$ and set
\[
A_{\ell, i} = \{v\in \PP^{2}_\QQ(\QQ)\mid g_\ell(v) = 0,\ \text{ and } H(v) = i\},
\]
and $n_{\ell,i} = \# A_{\ell,i}$. Then by~\Cref{thm:BCN-d}, we have that
\begin{equation}\label{eq:projecurve}
\sum_{1\leq i\leq k}n_{\ell,i}\ll d_{\ell}^{2}k^{\frac{2}{d_\ell}}(\log k)^{O(1)}.
\end{equation}
At the same time, for each $i\ge 1$ the Schwartz--Zippel bound gives us
\begin{equation}\label{bound:S-Zn-i}
n_{\ell,i} \ll d_\ell i.
\end{equation}
By our discussion,
$$
N_{\rm aff}(X  \cap ( \cup_{L \in I}  L   )  ,B) \leq (\#I_1)  +d(d-1)\sum_\ell \sum_{i=1}^{2B}n_{\ell,i}\dfrac{4B}{i}.
$$
It is now left to bound $\sum_{i=1}^{2B} n_{\ell,i}\frac{B}{i}$, for each $\ell$.
Let us first assume that $f_d$ is  $2$-irreducible, so that $d_\ell\geq 3$ for every $\ell$. From Inequality (\ref{bound:S-Zn-i}) we find,
$$
\sum_{i=1}^{d_\ell}n_{\ell,i}\dfrac{4B}{i}
\ll \sum_{i=1}^{d_\ell}d_\ell i\dfrac{4B}{i}
\ll  \sum_{i=1}^{d_\ell} d_\ell B   \le  d_\ell^2 B.
$$
On the other hand, summation by parts
yields:
\begin{align}
\sum_{i=d_\ell}^{2B}n_{\ell,i} \dfrac{4B}{i}&=\sum_{k=d_\ell}^{2B-1}\left(\sum_{i=d_\ell}^{k}n_{\ell,i} \right) \left(\dfrac{4B}{k}-\dfrac{4B}{k+1}\right)+\left(\sum_{i=d_\ell}^{2B}n_{\ell,i} \right)\dfrac{4B}{2B}\\
 &\ll d_\ell^2\left( \left( \sum_{k=d_\ell}^{2B-1}k^{\frac{2}{d_\ell}} \log(k)^{O(1)} \dfrac{B}{k(k+1)} \right ) +B^{\frac{2}{d_\ell}} \log(B)^{O(1)} \right). \label{eq:parts2}
\end{align}
Now since $d_\ell\geq 3$, one has (note that $d_\ell^{1/d_\ell}$ is universally bounded):
 \begin{equation}\label{eq:sumg}
 \sum_{k\geq d_\ell} \frac{ k^{\frac{2}{d_\ell}} \log(k)^{O(1)} }{k(k+1)}\ll  \sum_{k\geq d_\ell}  \frac{ \log(k)^{O(1)}}{{k^{2-2/d_\ell}} } \ll  \frac{d_\ell^{2/d_\ell} \log(d_\ell)^{O(1)}} { d_\ell }
  \ll   d_\ell^{-1} \log(d_\ell)^{O(1)},
\end{equation} 
and hence one finds
$$
\sum_{i=d_\ell}^{2B}n_{\ell,i}\dfrac{4B}{i} \ll d_\ell\log(d_\ell)^{O(1)}B + d_\ell^2B^{2/d_\ell}\log(B)^{O(1)} \ll d_\ell^2B.
$$
Thus we conclude the correct upper bound in this case, proving the proposition when $f_d$ is $2$-irreducible.
When $f_d$ is only $1$-irreducible (instead of $2$-irreducible), one estimates (\ref{eq:parts2}) by using
$$
  \sum_{k = 1 }^{2B-1} \frac{\log(k)^{O(1)}}{k+1}\ll    \log (B)^{O(1)}
$$
instead of (\ref{eq:sumg}).
The proposition is now fully proved.
\end{proof}

\begin{remark}
\label{rem: smooth surfaces have a small amount of lines}
Let $f \in \ZZ[x_1,x_2,x_3]$ be a polynomial of degree $d \ge 3$. If the surface defined in $\mathbb{P}^3$ by the homogenization of $f$ is smooth, 
then there can be at most $cd^2$ lines lying on $X$, where $c$ is a universal constant (see~\cite{Se43,BS07}). One therefore gets an improved bound of  $cd^2 B$ in Proposition \ref{prop1:Brow-Heath-Salb} under that assumption.
\end{remark}

We can now prove \Cref{prop:0.4strong}.

\begin{proof}[Proof of Proposition \ref{prop:0.4strong}] 
We start by using \Cref{prop:aux.affine.poly} on $f$ to find an auxiliary polynomial $g$ of degree at most $O(1) d^{5/2} B^{1/\sqrt{d}}(d+\log(B))$ which is
not divisible by $f$, and which vanishes on all integral points of $V(f)$ of height at most $B$.
We now consider two cases, depending on whether $d$ or $B$ is large.

Assume first that $d > \log(B)^2$.  Then $B^{1/\sqrt{d}} = O(1)$ is universally bounded, hence $g = f = 0$ is a (possibly reducible) curve of degree at most
\[
\ll d^{9/2}.
\]
The Schwartz--Zippel bound then gives that
\[
N_{\rm aff}(f,B) \leq N_{\rm aff}(C, B)\ll d^{9/2} B.
\]

Now assume that $d \leq \log (B)^2$. Hence, $g=f=0$ is a (possibly reducible) curve $C$ of degree at most
\[
\deg(C) \ll d\cdot d^{5/2} B^{1/\sqrt{d}}(d+\log(B))= B^{1/\sqrt{d}}\log(B)^{O(1)}.
\]
By Proposition~\ref{prop1:Brow-Heath-Salb}, the contribution of the union of all linear irreducible components of $C$ is at most $O(1)d^{4}B$ if $f_d$ is $2$-irreducible, and at most $O(1)d^4 B\log (B)^{O(1)}$ if $f_d$ is $1$-irreducible (instead of $2$-irreducible). By Corollary \ref{cor:counting.affine.curve}, the total contribution of components $C_i$ of degree $ 2 \leq \deg(C_i) \leq\log B$, is bounded by
\begin{align*}
\ll \deg(C)(\log B)^3 B^{1/2}(\log (B)+\log B) &\ll   B^{1/\sqrt{d}}B^{1/2}(\log B)^{O(1)}.
\end{align*}
Distinguishing between $d\geq 5$, $d = 4$, and $d=3$, this quantity is bounded by
\begin{align*}
&\ll B, & \text{ if } d \geq 5, \\
&\ll B \log(B)^{O(1)}, & \text{ if } d =4, \\
&\ll B^{1/2 + 1/\sqrt{3}} \log(B)^{O(1)}, & \text{ if } d =3. \\
\end{align*}
Finally, for each irreducible component $C_i$ of $C$ with $d_i:=\deg(C_i) > \log(B)$, we have by~\Cref{thm:BCN-d} that
\[
N_{\rm aff}(C_i, B) \ll d_i^2 B^{1/d_i}(\log B)^{O(1)} \ll d_i^2 B^{1/\log(B)} \log(B)^{O(1)} \ll d_i^2 \log(B)^{O(1)}.
\]
Summing over all irreducible components $C_i$ of $C$ with $d_i:=\deg(C_i) > \log(B)$,
\begin{flalign*}
\sum_i N_{\rm aff}(C_i, B)
&\ll \sum_i d_i^2  \log(B)^{O(1)}
\le \log(B)^{O(1)}\Big(\sum_i d_i\Big)^2 \\
&\le \log(B)^{O(1)}\deg(C)^2
\ll  B^{2/\sqrt{d}} \log(B)^{O(1)}.
\end{flalign*}
Again, this is $O(B)$ for $d\geq 5$, while it is also bounded as desired for $d=3$ and $d=4$. This concludes the proof.
\end{proof}


\subsection{Counting on linear spaces}\label{sec:linear:sub}

In this section we prove a higher dimensional variant of \Cref{prop1:Brow-Heath-Salb} for counting integral points on $(n-2)$-spaces lying on an affine hypersurface $X\subset \AA^n$. This will be important for  obtaining an exponent of $d$ in \Cref{thm:fd:Q} which is independent of $n$. We denote by $X_\sing$ the singular locus of $X$, and by $X_\ns = X\setminus X_\sing$ the smooth locus of $X$.

\begin{prop}\label{prop:count.on.n-2.planes}
Let $n\geq 4$ and let $f\in \ZZ[x_1, \ldots, x_n]$ be an NCC polynomial of degree $d\geq 2$ which is irreducible over $\QQ$. Let $I$ be a finite collection of $(n-2)$-planes on $X=V(f)$. Then for each  integer $B \ge 2$ we have
\begin{align*}
N_{\rm aff}(X  \cap ( \cup_{L \in I}  L   )  ,B) \ll_n d^7 B^{n-2} + \#I B^{n-3}, \quad &\text{ if $f_d$ is $2$-irreducible}, \\
N_{\rm aff}(X  \cap ( \cup_{L \in I}  L   )  ,B) \ll_n d^7 B^{n-2}\log(B) + \#I B^{n-3}, \quad &\text{ if $f_d$ is $1$-irreducible}.
\end{align*}
\end{prop}
We first need several lemmas.
\begin{lem}\label{lem:number.of.planes.through.point}
Let $n \ge 3$.
Let $X\subset \PP^n$ be a geometrically irreducible hypersurface of degree $d\geq 2$ and let $x\in X$. Assume that $X$ is not a cone with cone point $x$. Then $X$ contains at most $d^2$ linear spaces of dimension $(n-2)$ through $x$. If $x$ is furthermore a smooth point of $X$, this number is at most $d$.
\end{lem}

\begin{proof}
Any such plane must be contained in the tangent cone $T_x$ of $X$ at $x$. If $X$ is not a cone with cone point $x$, then the intersection $T_x\cap X$ is of dimension $n-2$ and degree at most $d^2$. If $x$ is a smooth point, then the tangent cone is simply the tangent space, and the intersection $T_x\cap X$ is of degree $d$.
\end{proof}

In the situation of the proposition, we say that $X = V(f)$ is \emph{$(n-2)$-ruled} if $X$ is the union of all $(n-2)$-planes contained in it. Equivalently, $X$ contains infinitely many $(n-2)$-planes. We first bound the number of $(n-2)$-planes on $X$ which do not come from an $(n-2)$-ruling of $X$. Let $X\subset \PP^n$ be a hypersurface and consider the Fano variety
\[
F_{n-2}(X) := \{L\in \Gr(n-1,n+1)\mid L\subset X\}
\]
of $(n-2)$-planes in $\PP^n$ contained in $X$. If $X$ is irreducible of degree $d\geq 2$ then the dimension of $F_{n-2}(X)$ is either $0$ or $1$, as a standard incidence correspondence shows. If $C$ is a one-dimensional irreducible component of $F_{n-2}(X)$, then the union of all $(n-2)$-planes from $C$ will cover $X$.

We first bound the number of $0$-dimensional irreducible components of $F_{n-2}(X)$, which are the $(n-2)$-planes on $X$ which do not belong to an $(n-2)$-ruling of $X$. We will call such a plane an \emph{isolated $(n-2)$-plane}. For this, we need some height theory over function fields, see e.g.\ \cite{lang-dioph}. Let $K$ be a function field of a smooth, projective, geometrically integral curve $C$ over $\overline{\QQ}$. For $P$ in $C(\overline{\QQ})$ let $v_P: K^\times\to \ZZ$ be the valuation corresponding to $P$. Using these valuations, we define a height function as follows. If $\tilde{f} = (f_0 : \ldots : f_n)$ is an element of $\PP^n(K)$, then let
\[
h_K(\tilde{f}) = \sum_{P\in C(\overline{\QQ})} \max_{i=0, \ldots, n}\{-v_P(f_i)\}.
\]
As usual, the product formula shows that this is well-defined. If $L/K$ is a finite extension, and $\tilde{f}\in \PP^n(L)$, then $L$ corresponds to a finite cover $D\to C$ of $C$, which we may assume to also be smooth and projective, and we define $h_L(\tilde{f})$ the same way. To ensure compatibility, we normalize this by
\[
h(\tilde{f}):=\frac{1}{[L:K]}h_L(\tilde{f}).
\]
For $f\in \overline{K}$ we let $h(f) = h(1:f)$. Note that if $f\in K$ has height $e$, then $f$ has at most $e$ zeroes on $C$. In general, if $C'$ is a geometrically integral curve, then we use the projective smooth model of $C'$ to define the height function on the function field of $C'$. Such a model exists by e.g.~\cite[Sec.\,I.6, Cor.\,6.11]{hartshorne}.

For a polynomial $F = \sum_i a_i X^i$ over $K$ in any number of variables, let $v_P(F) = \min_i v_P(a_i)$. We define $h_K(F)$ to be the height of the point consisting of the coordinates of $F$ in projective space. Gauss' lemma shows that $v_P(FG) = v_P(F) + v_P(G)$ and hence we obtain also that $h_K(FG) = h_K(F) + h_K(G)$. Also, if $(f_i)_i$ are elements of $K$ and if $S\subset C(\overline{\QQ})$ is a finite set such that all poles of the $f_i$ are contained in $S$, then $h_K(\sum_i f_i) \leq \#S \cdot \max_i \{h_K(f_i)\}$.

\begin{lem}\label{lem:number.of.n-2.planes}
Let $n \ge 3$, and let $X\subset \PP^n$ be a geometrically integral hypersurface of degree $d\geq 2$. Then $X$ contains at most $2d^7$ isolated $(n-2)$-planes.
\end{lem}

The idea for this lemma is as follows. Take a generic $2$-plane $P$ in $\PP^n$, let $C$ be the intersection $X\cap P$, and  assume that $C$ is geometrically integral. Any $(n-2)$-plane $H$ on $X$ intersects $C$ in a unique point $x$, and if this point is smooth then $H$ is contained in the tangent space $T_x X$. The idea is now that for a point $x\in C$, the intersection $X\cap T_xX$ may be described via a polynomial $g\in \overline{\QQ}(C)[a_1, \ldots, a_n]$. We control how this polynomial factors at the points of $C$ by using Noether forms, from which the result follows.

\begin{proof}[Proof of \Cref{lem:number.of.n-2.planes}]
The singular locus $X_\sing$ is contained in a subvariety of $\PP^n$ of (pure) dimension $n-2$ and degree $d^2$, and hence contains at most $d^2$ linear spaces of dimension $(n-2)$. Let $Z$ be the union of all isolated $(n-2)$-planes on $X$ which are not contained in $X_\sing$, and denote by $Y$ the closed set $Z\cap X_\sing$. This is an algebraic set of dimension at most $n-3$. We now take a generic $2$-plane $P\subset \PP^n$ such that $P\cap Y = \emptyset$, such that $P\cap X$ is geometrically integral, such that $P$ is not contained in the hyperplane $\{x_0 = 0\}$, and such  that the intersection of $\{x_0 = 0\}$ and $P$ does not contain any singular points of $X$. If $L$ is an isolated $(n-2)$-plane on $X$ which is not contained in $X_\sing$, then by construction $L$ intersects $P$ in a unique point in $X_\ns$. So it suffices to count the number of $x\in X_\ns\cap P$ for which there exists an isolated $(n-2)$-plane on $X$ through $x$. Indeed, \Cref{lem:number.of.planes.through.point} shows that the total number of isolated $(n-2)$-planes contained in $X$ is at most $d$ times this quantity plus $d^2$ (coming from $(n-2)$-planes in $X_\sing$).

Now take $x\in X_\ns\cap P$ and assume that there is some $(n-2)$-plane $L$ on $X$ through $x$. Then $L$ must also be contained in the tangent space $T_x X$. Let $f(x_0, \ldots, x_n)\in \overline{\QQ}[x_0, \ldots, x_n]$ be a defining equation for $X$, and assume without loss of generality that $\frac{\partial f}{\partial x_0} (x) \neq 0$. Define the polynomial
\[
g(x,a) := f\left(-\frac{\partial f}{\partial x_1}(x) a_1 - \ldots -\frac{\partial f}{\partial x_n}(x) a_n, \frac{\partial f}{\partial x_0}(x) a_1, \ldots, \frac{\partial f}{\partial x_0}(x) a_n\right)
\]
in $\overline{\QQ}[x_0, \ldots, x_n][a_1, \ldots, a_n]$. Then $X$ contains an $(n-2)$-plane through $x$ if and only if the polynomial $g(x,a)\in \overline{\QQ}[a_1, \ldots, a_n]$ has a linear factor in $\overline{\QQ}[a_1, \ldots, a_n]$. Let $K$ be the function field of the geometrically integral curve $C=P\cap X_{\ns}$ (over $\overline{\QQ}$). By multiplying with some negative power of $x_0$, we can consider $g$ as an element of $\overline{K}[a_1, \ldots, a_n]$. Each coefficient of $g$ is of degree at most $d(d-1)$ in $\overline{\QQ}[x]$, and hence has height at most $d^2(d-1)$ in $K$. We factor $g$ in this ring into irreducible polynomials $g_i\in \overline{K}[a_1, \ldots, a_n]$. The number of $g_i$ which are linear is equal to the number of one-dimensional irreducible components of $F_{n-2}(X)$. Since we are only considering isolated $(n-2)$-planes, we can disregard these and assume that $\deg g_i = \delta_i \geq 2$ for all $i$. We fix $i$ for the moment, and let $\{F_\ell\}_\ell$ be a collection of Noether forms for polynomials of degree $\delta_i$ from \Cref{thm:noether.forms}, for which we have that $\deg F_\ell\leq \delta_i^2$. Since $g_i$ is irreducible over $\overline{K}$, not all $F_\ell$ vanish when applied to $g_i$. So let $F_\ell$ be non-vanishing on $g_i$ and denote by $\tilde{g}\in K$ the product of all conjugates of $F_\ell(g_i)$. We bound the height of this $\tilde{g}$, which will give us an upper bound on its number of zeroes. By the discussion above, we have that $\sum_i h(g_i) \leq h(g)$ and hence $h(g_i)\leq h(g)$. Now note that all poles of all coefficients of $g$ are contained in the intersection of $C$ with the line defined by $x_0 = 0$ in $P$, which consists of at most $d$ points. By Gauss' lemma, all poles of all coefficients of the $g_i$ therefore lie above these points, and so we obtain that $h(F_\ell(g_i))\leq d \delta_i^2 d^2(d-1)$. Hence $h(\tilde{g})\leq \delta_i d^5$. Now sum over all $i$ to conclude that there are at most $d^6$ points $x$ in $C$ for which $g$ has a linear factor. For each such point, there can be at most $d$ isolated $(n-2)$-planes on $X$ containing this point, and hence we conclude that $X$ contains at most $d^7 + d^2$ isolated $(n-2)$-planes.
\end{proof}

If $X\subset \AA^n$ is an affine hypersurface of degree $d$ defined by a polynomial $f\in \QQ[x_1, \ldots, x_n]$, then we denote by $X_\infty\subset \PP^{n-1}$ the intersection of $X$ with the hyperplane at infinity. This is a projective variety of degree $d$, defined by $f_d=0$.

\begin{lem}
Let $X\subset \AA^n$ be a geometrically irreducible affine hypersurface.
 If $X$ is $(n-2)$-ruled then every irreducible component of $X_\infty$ is $(n-3)$-ruled.
\end{lem}

\begin{proof}
If $X$ is $(n-2)$-ruled, then so is its projective closure $\overline{X}$ in $\PP^n$ as it contains infinitely many $(n-2)$-planes. It follows $X_\infty$ is $(n-3)$-ruled. Let $Y$ be an irreducible component of $X_\infty$ and take $y \in Y$ which does not lie on any other irreducible components of $X_\infty$. There exists an $(n-3)$-plane contained in $X_\infty$ and passing through $y$, which must then lie on $Y$ by irreducibility. The claim now follows as $Y$ must contain infinitely many $(n-3)$-planes.
\end{proof}

Whenever we talk about the degree of a subvariety of a Grassmanian, we mean its degree after composing with the Pl\"ucker embedding.

\begin{proof}[Proof of \Cref{prop:count.on.n-2.planes}]
Recall that $X=V(f) \subset \AA^n_{\QQ}$.
By \Cref{lem:number.of.n-2.planes}, the number of isolated $(n-2)$-planes on $X$ is at most $O(d^7)$. These have at most $O_n(d^7B^{n-2})$ integral points up to height $B$, and so we can focus on the one-dimensional components of $F_{n-2}(X)$. In particular, we can assume that $X$ is $(n-2)$-ruled.

Let $Y$ be an irreducible component of $X_\infty$ of degree $\delta$, which is $(n-3)$-ruled by the previous lemma. So the Fano scheme $F_{n-3}(Y)$ of $(n-3)$-planes on $Y$ is of dimension one. Let $C\subset F_{n-3}(Y)$ be an irreducible component of $F_{n-3}(Y)$ of dimension one which $(n-3)$-rules $Y$. We claim that
\[
\delta \leq \deg C.
\]
To prove this, let $\ell$ be a generic line in $\PP^{n-1}$ for which $\ell\cap Y$ consists of $\delta$ points. Then
\[
H := \{L\in \Gr(n-2, n)\mid L\cap \ell \neq \emptyset \}
\]
is a hyperplane in $\Gr(n-2,n)$. We claim that $C$ is not contained in $H$. Suppose towards a contradiction that $C\subset H$. Then the regular map
\[
C\to \ell: L\mapsto L\cap \ell
\]
is well-defined, and the image is the finite set $\ell\cap Y$ as $C$ rules $Y$. But $C$ is irreducible, and so the image is a single point, contradicting that $\ell \cap Y$ consists of $\delta\geq 2$ points. So $C$ is not contained in $H$ and $\#H\cap C\leq \deg C$. Therefore
\[
\delta = \#\ell\cap Y\leq \#H \cap C \leq \deg C,
\]
proving the lower bound.

Conversely, we claim that the sum of the degrees of the one-dimensional components of $F_{n-3}(Y)$ is at most $\delta^2$. Denote this union of all one-dimensional components of $F_{n-3}(Y)$ by $D$. To prove the claim, we take a generic hyperplane $H$ in $\Gr(n-2,n)$, which corresponds to a line $\ell$ in $\PP^{n-1}$ as above. Then $\ell$ is also generic (it is not contained in $Y$ as $H \cap C$ is finite, and $Y$ is ruled), and so we can assume that $\ell$ intersects $Y$ in $\delta$ smooth points. If $L\in H\cap C$ then we obtain an $(n-3)$-plane on $Y$ through a point on $\ell\cap Y$. By \Cref{lem:number.of.planes.through.point} there are at most $\delta^2$ such planes, so that
\[
\#H\cap D  = \deg D\leq \delta^2.
\]
We have now proven the claim.

If $L\subset X$ is an $(n-2)$-plane, then (see e.g.~\cite[Lemma 4.5]{Browning-Q})
\[
N_{\rm aff}(L, B)\ll_n \frac{B^{n-2}}{H(L)} + B^{n-3},
\]
where $H(L)$ is the height of the corresponding point in projective space coming from the Pl\"ucker embedding of $L\cap \PP^{n-1}\in \Gr(n-2, n)$. If $H(L)> B$ then $N_{\rm aff}(L,B)\ll_n B^{n-3}$, and so it suffices to treat those $L$ with $H(L)\leq B$. Each of these gives a rational point on $F_{n-3}(Y)$ of height at most $B$, for some component $Y$ of $X_\infty$. So using \Cref{lem:number.of.planes.through.point} the count is
\begin{align*}
N_{\rm aff}(X  \cap ( \cup_{L \in I}  L   )  ,B) &\ll_n \#I B^{n-3} + d^2 \sum_{C} \sum_{\substack{L\in C(\QQ)\\ H(L)\leq B}} \frac{B^{n-2}}{H(L)},
\end{align*}
where the sum is over all $C$ which are one-dimensional components of $F_{n-3}(Y)$ for some component $Y$ of $X_\infty$. Now one proceeds exactly as in the proof of \Cref{prop1:Brow-Heath-Salb} to conclude.
\end{proof}

\subsection{Proofs of our main results}

When cutting with hyperplanes to induct on the dimension in the proof of \Cref{thm:fd:Q}, we need a way to preserve the NCC condition and $2$-irreducibility in the case where $d$ is small. This is provided by the following corollary. We postpone the proof to \Cref{sec:H.condition}, where we immediately prove it over global fields (see \Cref{cor:preserve.H.cond.with.HIT}).

\begin{cor}\label{cor:preserve.H.cond.with.HIT.Q}
Let $f \in \ZZ[x_1, \ldots, x_n]$ be of degree  $d\geq 2$ with $n\geq 4$. Assume that $f$ is NCC and that $f_d$ is  $r$-irreducible over $\QQ$ for some integer $r \ge 1$. Then there exist linearly independent linear forms $\ell, \ell'$ of height at most $O_n(d^3)$ and an integer $t$ with $|t|\leq O_n(d^3)$ such that $f_d|_{V(\ell- t\ell')}$ is $r$-irreducible, and such that there are at most $O_n(d)$ values $b\in \QQ$ for which $f|_{V(\ell- t\ell'- b)}$ is not NCC.
\end{cor}

We can now prove \Cref{thm:fd:Q} and \Cref{thm:fd:Q:conj}.

\begin{proof}[Proof of \Cref{thm:fd:Q}]
Recall that we are given $f\in \ZZ[x_1, \ldots, x_n]$ of degree $d\geq 3$, and we wish to bound the number of integral points $x$ of height at most $B$ satisfying the equation $f(x)=0$. We use induction on $n$, where the base case $n=3$ is given by Proposition~\ref{prop:0.4strong}. We therefore assume that $n \ge 4$, and let $\kappa = 5(n-2) = O_n(1)$, where $5n-12=5(n-2)-2$ is the constant from~\Cref{prop:pila} when $m=n-2$.

First assume that $d>\kappa^{n-1}$. Let $g$ be the auxiliary polynomial from~\Cref{prop:aux.affine.poly} applied to $f$, it has degree at most
\[
\ll_n d^{4-\frac{1}{n-1}} B^{\frac{1}{d^{1/(n-1)}}} \log B,
\]
is coprime to $f$ and vanishes on all integral points of $f=0$ of height at most $B$. Let $C$ be the variety cut by $f=g=0$, it is of dimension $n-2$ and has degree at most $O_n(1) d^{5-\frac{1}{n-1}} B^{\frac{1}{d^{1/(n-1)}}} \log B$. Let $C_i$ be the irreducible components of $C$, and denote by $d_i$ the degree of $C_i$.

Assume that $d < (\log B)^{n-1}$. By \Cref{prop:count.on.n-2.planes} the contribution to $N_{\rm aff}(f,B)$ coming from the components $C_i$ with $d_i = 1$ is at most $O_n(1) d^7 B^{n-2}$ if $f_d$ is $2$-irreducible and at most $O_n(1) d^7 B^{n-2} \log (B)$ if $f_d$ is $1$-irreducible (instead of $2$-irreducible). Using~\Cref{prop:pila}, the contribution to $N_{\rm aff}(f, B)$ coming from components $C_i$ with $2\leq \delta_i\leq \log B$ is at most
\[
O_n(1) d^{5-\frac{1}{n-1}}(\log B)^\kappa B^{n-3 + \frac{1}{2} + \frac{1}{d^{1/(n-1)}}}.
\]
By our assumption that $d> \kappa^{n-1}\geq 3^{n-1}$ this is indeed bounded by $O_n(1) B^{n-2}$. On the other hand, if $\delta_i > \log B$ then $B^{\frac{1}{\log B}} = O(1)$ and so using~\Cref{prop:pila}, the contribution from components $C_i$ with $\delta_i > \log B$ is at most
\[
O_n(1) B^{n-3} (d^{5-\frac{1}{n-1}})^\kappa B^{\frac{\kappa}{d^{1/(n-1)}}}.
\]
By our assumption that $d > \kappa^{n-1}$, this quantity is also bounded by $O_n(1) B^{n-2}$.

Now assume that $d > (\log B)^{n-1}$, but still that $d > \kappa^{n-1}$. Then the Schwartz--Zippel bound immediately gives that
\[
N_{\rm aff}(C, B) \ll_n d^{5-\frac{1}{n-1}} B^{\frac{1}{d^{1/(n-1)}}}B^{n-2} \ll_n d^{5 - \frac{1}{n-1}} B^{n-2}
\]
since $B^{1/d^{1/(n-1)}}$ is bounded by a universal constant.

Finally, assume that $d \leq \kappa^{n-1} = O_n(1)$.
By \Cref{cor:preserve.H.cond.with.HIT.Q}, we may find linearly independent $\QQ$-linear forms $\ell,\ell'$ such that after a change of coordinates  $x_n:=\ell(x)$, $x_{n-1}:=\ell'(x)$ of height $O_n(d^3) = O_n(1)$, there exists an integer $t$ of height at most $O_n(d^3)=O_n(1)$ such that the polynomial $f_d(x_1, \ldots, x_{n-1}, tx_{n-1})$ is $2$-irreducible (or $1$-irreducible if $f_d$ was) and such that there are at most $O_n(d) = O_n(1)$ values $b\in \QQ$ for which $f(x_1, \ldots, x_{n-1}, tx_{n-1} + b)$ is not NCC. Now we use the induction hypothesis for the good values of $b$ and the Schwartz--Zippel bound for the bad values of $b$ to get
\begin{align*}
N_{\rm aff}(f, B)&\leq \sum_{|b|\leq (|t|+1)B} N_{\rm aff}\left( \{f=0\}\cap \{x_n = tx_{n-1} + b\}, B\right) \\
	&\ll_n  d^3B(d^5B^{n-3}) + d^2B^{n-2} \ll_n B^{n-2}.
\end{align*}
If $f_d$ is $1$-irreducible, the induction hypothesis introduces an additional factor of the form $\log(B)^{O_n(1)}$ in the upper bound above.
\end{proof}

\begin{proof}[Proof of \Cref{thm:fd:Q:conj}] 
Let $f\in \ZZ[x_1, \ldots, x_n]$ be of degree $d$ such that $f_d$ is $1$-irreducible. We may assume that $f$ is irreducible. We induct on $n$, where the base case $n=2$ is~Conjecture \ref{conj:uniform.curve.conjecture}. So assume that $n\geq 3$. First, if $\lVert f\rVert$ is larger than $O_n\left(B^{d\binom{d+n-1}{n-1}})\right)$, then we can use~\cite[Lem.\,5]{Brow-Heath-Salb} to bound $N_{\rm aff}(f, B)$ as desired. So we may assume that $\log \lVert f\rVert \ll_n d^{O_n(1)}\log B$. By~\Cref{cor:eff:Hil} with $r=m=s=1$, we can find $t\in \ZZ$ with $|t| \leq \mathrm{poly}_{n,d}(\log \lVert f\rVert) \le \mathrm{poly}_{n,d}(\log B)$ such that $f_d(x_1, \ldots, x_{n-1}, tx_{n-1})$ is still $1$-irreducible. By induction, we have
\begin{align*}
N_{\rm aff}(f, B)&\leq \sum_{|b|\leq (|t|+1)B} N_{\rm aff}\left( \{f=0\}\cap \{x_n = tx_{n-1} + b\}, B\right) \ll_{n,d,\varepsilon} B^{n-2+\varepsilon}. \qedhere
\end{align*}
\end{proof}

The projective form of \Cref{thm:dgcdegree:proj} follows, as usual, from the affine case given by \Cref{thm:fd:Q}, 
where we sharpen the exponent of $d$, compared to Theorem 4 and Remark 4.3.8 of \cite{CCDN-dgc} (we amend Remark 4.3.8 of \cite{CCDN-dgc}: the expressions in $n$ are inaccurate, according to the amendments to Proposition~4.3.4 and Lemma 4.3.7 of \cite{CCDN-dgc}; this correction is now obsolete by our improvements provided by Theorems \ref{thm:dgcdegree:proj} and \ref{thm:fd:Q}). 



\begin{proof}[Proof of \Cref{thm:dgcdegree:proj}]
Consider the affine cone
 $X_{\rm aff}$ over $X$, which is  an affine hypersurface inside $\AA^{n+1}_{\QQ}$ (given by the same equation as $X$), and use \Cref{thm:fd:Q} for $X_{\rm aff}$. We are done since clearly
$
N(X,B)\le N_{\rm aff}(X_{\rm aff},B).
$ 
\end{proof}

\section{Global fields}\label{sec:global}

\subsection{Definitions and main results}

In this final section we show variants of most of the results from \Cref{sec:intro} for all global fields $K$ (including the positive characteristic case), with the results of the previous sections being for $K=\QQ$. In the context of global fields, uniform upper bounds on rational and integral points on varieties were first obtained over $K = \FF_q(t)$ for curves by Sedunova~\cite{Sedunova} by adapting the Bombieri--Pila method~\cite{bombieri-pila}.  For large characteristics, Cluckers, Forey and Loeser~\cite{CFL} improve upon this using model theoretic tools. In higher dimensions, the first results were obtained by Vermeulen~\cite{Vermeulen:p}, who proved the dimension growth conjecture for hypersurfaces of large degree over $\FF_q(t)$. Paredes and Sasyk~\cite{Pared-Sas} generalized this to obtain dimension growth for arbitrary projective varieties of degree $d\geq 4$ over global fields.
As in the above sections, our work improves the dependence on $d$ (see \Cref{thm:dgcdegree:proj:K}), and, more importantly, generalizes the affine situation (see \Cref{thm:fd:K,thm:final}).

We fix a global field $K$, i.e.\ a finite extension of $\QQ$ or $\FF_q(t)$ for some prime power $q$. If $K$ is a finite extension of $\FF_q(t)$ we moreover assume that $K$ is separable\footnote{Note that any global field of positive characteristic is a finite separable extension of a field isomorphic to $\FF_q(t)$ for some prime power $q$, by the existence of separating transcendental bases.} over $\FF_q(t)$ and that $\FF_q$ is the full field of constants of $K$. We denote by $k$ either $\QQ$ or $\FF_q(t)$ depending on whether $K$ is an extension of $\QQ$ or of $\FF_q(t)$. Let $d_K = [K:k]$.

Let us recall the theory of heights on $K$. We follow the same normalization as in~\cite{Pared-Sas}. Assume first that $K$ is a number field. The infinite places on $K$ come from embeddings $\sigma: K\to \CC$ which give a place $v$ via
\[
|x|_v :=  |\sigma(x)|^{n_v/d_K},
\]
where $|\cdot |$ is the usual absolute value on $\RR$ or $\CC$, and $n_v = 1$ if $\sigma(K)\subset \RR$ and $n_v = 2$ otherwise. Denote by $M_{K, \infty}$ the set of infinite places of $K$. The finite places of $K$ correspond to non-zero prime ideals of the ring of integers $\cO_K$. For such a prime ideal $\frak{p}$ we obtain a place $v$ via
\[
|x|_v := N_K(\frak{p})^{-\ord_{\frak{p}}(x) / d_K},
\]
where $N_K(\frak{p}) = \# \cO_K / \frak{p}$ is the norm of $\frak{p}$. Denote by $M_{K,\mathrm{fin}}$ the set of finite places of $K$.

Now assume that $K$ is a function field. Any place of $K$ then corresponds to a discrete valuation ring $\cO$ of $K$ containing $\FF_q$ and whose fraction field is $K$. Let $\frak{p}$ be the maximal ideal of $\cO$. Then we define a place $v$
\[
|x|_v := N_K(\frak{p})^{-\ord_{\frak{p}}(x) / d_K},
\]
where as before $N_K(\frak{p}) = \# \cO/\frak{p}$. Denote by $M_K$ the set of places of $K$. We also fix a place $v_\infty$ above the place in $\FF_q(t)$ defined by $|f|_\infty := q^{\deg f}$. The corresponding prime is denoted by $\frak{p}_\infty$. Let $M_{K, \infty} = \{v_\infty\}$ and let $M_{K, \mathrm{fin}} = M_K \setminus \{v_\infty\}$. The ring of integers of $K$ is the set of $x\in K$ for which $|x|_v\leq 1$ for all places $v\neq v_\infty$.

Let $K$ be an arbitrary global field again. With the above definitions, recall that we have a product formula, which states that
\[
\prod_{v\in M_K} |x|_v = 1,
\]
for all $x\in K$. For a point $(x_0 : \ldots : x_n)\in \PP^n(K)$ we define the height to be
\begin{flalign}
\label{eq: relative height}
H(x) := \prod_{v\in M_K} \max_i |x_i|_v,
\end{flalign}
which is well-defined by of the product formula. If $x\in K$ let $H(x) := H(1:x)$. We also consider the relative height defined as
\[
H_K(x) := H(x)^{d_K}.
\]
If $X\subset \PP^n$ is a projective variety we define
\[
N(X,B)
\]
to be the number of points $x$ on $X(K)$ with $H_K(x)\leq B$. In the affine setting, for an integer $B\geq 1$ let $[B]_{\cO_K}$ be the elements $x\in \cO_K$ such that
\[
\begin{cases}
\max\limits_{\sigma: K\hookrightarrow \CC} |\sigma(x) | \leq B^{1/d_K}, & \text{ if $K$ is a number field, or} \\
|x|_{v_\infty} \leq B^{1/d_K}, & \text{ if $K$ is a function field}.
\end{cases}
\]
We note that $\#[B]_{\cO_K} \ll_K B$.
If $X\subset \AA^n$ is an affine variety then we define
\[
N_{\rm aff}(X,B)
\]
to be the number of points in $X(K)\cap [B]_{\cO_K}^n$. These definitions ensure that if $x\in \PP^n_K(K)$ with $H_K(x)\leq B$, then there exists a point $y = (y_1, \ldots, y_{n+1})\in\AA^{n+1}_K(K)$ for which $(y_1 : \ldots : y_{n+1}) = x$ and such that $y\in [O_K(1) B]^{n+1}_{\cO_K}$, see \cite[Prop.\,2.2]{Pared-Sas}.

Our main result on dimension growth for affine hypersurfaces generalizes to the global field setting. Let $r\ge 1$ be an integer. We say that a polynomial $f$ in $K[y_1, \ldots, y_n]$ is \emph{$r$-irreducible over $K$} (abbreviated by $r$-irreducible if $K$ is clear), if $f$ does not have any factors of degree $\leq r$ over $K$. We say that $f$ is \emph{NCC over $K$} if there does not exist a $K$-linear map $\ell:\AA_K^n\to \AA_K^2$ and a curve $C$ in $\AA_K^2$ such that $V(f)=\ell^{-1}(C)$. Our main theorem in the generality of global fields is the following.

\begin{thm}\label{thm:fd:K}
Let $K$ be a global field with ring of integers $\cO_K$. Let $e = 7$ if $\charac K = 0$ and $e = 11$ if $\charac K > 0$. Given an integer $n\geq 3$, there exist constants $c = c(K,n)$ and $\kappa= \kappa(n)$ such that for all polynomials $f$ in $\cO_K[y_1, \ldots, y_n]$ of degree $d\geq 3$ such that $f$ is irreducible over $K$ and NCC, the following holds. If $f_d$ is $2$-irreducible over $K$, and its absolutely irreducible factors can be defined over the separable closure $K^{\mathrm{sep}}$ of $K$, one has for all $B\geq 2$ that
\begin{align*}
N_{\rm aff}(f, B)&\leq cd^e B^{n-2}, &\text{ if } d\geq 5, \\
N_{\rm aff}(f, B)&\leq c B^{n-2} (\log B)^\kappa, &\text{ if } d = 4, \\
N_{\rm aff}(f, B)&\leq c B^{n-3 + 2/\sqrt{3}} (\log B)^\kappa, &\text{ if } d = 3.
\end{align*}
If $f_d$ is only $1$-irreducible over $K$ (instead of $2$-irreducible), then we have that
\begin{align*}
N_{\rm aff}(f, B)&\leq cd^e B^{n-2} (\log B)^\kappa, &\text{ if } d \geq 4, \\
N_{\rm aff}(f, B)&\leq c B^{n-3 + 2/\sqrt{3}} (\log B)^\kappa, &\text{ if } d = 3.
\end{align*}
\end{thm}

\begin{remark}
In fact, in the final case of Theorem \ref{thm:fd:K} (with $f$ being NCC and irreducible over $K$ and $f_d$ being $1$-irreducible over $K$), one easily derives   that there is $\kappa=\kappa(n)$ and $c=c(K,n)$ such that
\begin{equation}\label{eqd2-4.1}
N_{\rm aff}(f, B)\leq cd^2 B^{n-2} (\log B)^\kappa, \text{ if } d \geq 4,   \\
\end{equation}
by furthermore applying  Theorem 6  \cite{BinCluKat}.
Indeed, by Theorem \ref{thm:fd:K} we may focus on the case with $d>\log B$, which follows immediately from Theorem 6 of \cite{BinCluKat} since $B^{1/d}$ is bounded by a constant in this case. Note that the quadratic dependence on $d$ in (\ref{eqd2-4.1}) is optimal, by \cite[Section 6]{CCDN-dgc}.
\end{remark}

\begin{remark}
Let us mention that in principle it is possible to make the dependence of $c$ on $K$ in the above result effective and explicit. If $\charac K > 0$, then this follows from the Riemann hypothesis for the function field $K$, and the resulting $c$ will only depend on the genus, degree, and field of constants of $K$. In characteristic zero, one can make the dependence explicit if one assumes the generalized Riemann hypothesis. We refer to~\cite[Rem.\,2.7]{Pared-Sas} for more information.
\end{remark}

\subsection{Effective Hilbert irreducibility}\label{sec:eff.hilbert}



In this section we prove versions of Theorems \ref{thm:eff.Hil.lemma:Q} and \ref{thm:HIT.v2:Q}  in the more general settings of global fields.
If $P$ is a polynomial, $H(P)$ is defined to be the height of the tuple of its coefficients, considered in projective space. We similarly define $H_K(P)$ as in Equation \ref{eq: relative height}.

\begin{thm}\label{thm:eff.Hil.lemma}
Let $r,n,d$ be positive integers,
 let $ \underline{T}=(T_1, \ldots, T_r),\underline{Y}=(Y_1, \ldots, Y_n)$ be tuples of variables, and let
$F\in \cO_K[\underline{T},\underline{Y}]$ be an irreducible polynomial of degree $d$ in  $K[\underline{T},\underline{Y}]$ such that $\deg_{\underline{Y}} F \geq 1$. For an integer $B \geq 1$ denote by $S_T(F, B)$ the number of $\underline{t}=(t_1,\ldots,t_r)\in [B]_{\cO_K}^r$ for which $F(\underline{t},\underline{Y})$ is reducible in $K[\underline{Y}]$. Then
\[
S_T(F, B) \ll_{K,r,n} 2^{c(r)(d+1)^n}  (\log H_K(F) + 1)^{10} B^{r- 1/2}\log(B)^{10(r-1)},
\]
where $c(r)=118+10(r-1)$.
\end{thm}
%
\begin{proof} We first prove the case $r=1$.  Set $T_1=T$, put $a = 1+\max_i \deg_{Y_i} F$ and consider the Kronecker transform of $F$
\[
\Kr(F) = F(T, Y, Y^a, \ldots, Y^{a^{n-1}})\in \cO_K[T,Y].
\]
Note that $H_K(\Kr(F)) = H_K(F)$, that $\deg_T(\Kr(F)) = \deg_T(F)$, and that
\[
\deg_{\underline{Y}} (\Kr(F)) \leq (a-1) + (a-1)a + \ldots (a-1)a^{n-1} = a^n - 1 < (d+1)^n.
\]
Let
\[
\Kr(F) = \prod_{i=1}^m Q_i(T,Y)
\]
be an irreducible factorization of $\Kr(F)$ in $K[T,Y]$. 
We have that $m\leq \deg_{Y} (\Kr (F))$ and for each $i$ that
\[
\deg_T (Q_i) \leq \deg_T (\Kr (F)), \quad \deg_{Y} (Q_i) \leq \deg_{Y} (\Kr (F)).
\]
Also, by combining \cite[Ch.3,\,Prop.\,2.4]{lang-dioph} and \cite[Lemma 2.1]{Pared-Sas:Hilbert}
we have
\[
H_K(Q_i) \ll_{K,n} H_K(Q_i) H_K(\prod_{j\not=i} Q_j) \ll_{K,n} 4^{d_K(\deg (\Kr(F))+1)^2} H_K(\Kr(F)).
\]
We now follow the proof of \cite[Lem.\,13.1.3]{fried-jarden} (see the note right after \cite[Equation (13.2)]{fried-jarden}),
which gives for each non-empty proper subset $I\subset \{1, \ldots, m\}$ a non-zero polynomial $c_I(T)\in K[T]$ such that $\deg c_I \leq \deg_T F$
  with the following property: let $t_0\in K$ and assume that for each $i = 1, \ldots, m$, the polynomial $Q_i(t_0, Y)$ is irreducible in $K[Y]$, and for each non-empty proper subset $I\subset \{1, \ldots, m\}$, we have $c_I(t_0)\neq 0$, then $F(t_0, Y_1, \ldots, Y_n)$ is irreducible in $K[Y_1, \ldots, Y_n]$.
In other words, we have that
\begin{flalign}
\tag{$*$}
S_T(F, B) \leq \sum_{i=1}^m S_T(Q_i, B) + \sum_{I \subset \{1,\ldots,m \}} \deg c_I,
\end{flalign}
where $\deg_{Y} Q_i \le \deg_{Y} \Kr(F) \le (d+1)^n$ and $\deg_T Q_i,\deg c_I \le d$.
We conclude using~\cite[Thm.\,1.1]{Pared-Sas:Hilbert} to bound $S_T(Q_i, B)$, recalling $m \le (d+1)^n$:
\[
(*)
 \ll_{K,n}  (d+1)^n2^{36(d+1)^n}(d+1)^{35n} d^{26}(\log H_K(F)+1)^{10}  (d+1)^{20n}B^{1/2}
+2^{(d+1)^n}d.
\]
%
To prove the multivariate version we proceed by induction, where the base case is $r=1$. Assume that $r>1$, and consider  specializations of $T_r$. Using the $r=1$ case we get,
\begin{flalign*}
S_{T}(F,B) \ll_K
S_{T_r}(F,B)B^{r-1}
+ \sum_{t_r \in [B]_{\cO_K}} S_{(T_1,\ldots,T_{r-1})}(F|_{T_r=t_r},B).
\end{flalign*}
Note that for every $t_r \in [B]_{\cO_K}$ we may bound the height of the specialization $F|_{T_r=t_r}$,
\[
H_K(F|_{T_r=t_r})
\ll_K (1+B+\ldots+B^d) H_K(F)
\ll_K B^d H_K(F).
\]
Using the induction hypothesis, one  verifies
\begin{flalign*}
S_T(F, B) \ll_{K,n,r}& 2^{c(1)(d+1)^n}(\log H_K(F) + 1)^{10} B^{1/2}B^{r-1} \\
&+ B 2^{c(r) (d+1)^n} (\log H_K(F)+1)^{10} B^{(r-1)-1/2} \log(B)^{10(r-1)},
\end{flalign*}
where $c(r)=118+10(r-1)$, arriving at the required bounds.
\end{proof}

%
%

We now prove a more precise version of \Cref{cor:eff:Hil} for global fields.

%
\begin{cor}\label{cor:eff:Hil.genVersion}
Let $r,n,m,d$ be positive integers,
let $\underline{T}=(T_1, \ldots, T_r), \underline{Y}=(Y_1, \ldots, Y_n)$ be tuples of variables, and let $F_1, \ldots, F_m$ be irreducible polynomials in $K[\underline{T}, \underline{Y}]$ of degree at most $d$ such that $\deg_{\underline{Y}}(F_i) \ge 1$ for each $i$.
Then there exists a polynomial $G$ of degree at most $30$ in $m$ variables whose  coefficients depend on $r,n,m,d$ such that   the following holds for every $s\ge 1$. If $X \subset \AA^r_K$ is a (possibly reducible) hypersurface of degree at most $s$, then
we have a tuple $\underline{t}\in \cO_K^r$ not contained in $X(K)$
 of height \[
 \ll_{K,n,r} s^{2}\log(s)^{20(r-1)}+G(\log H_K(F_1), \ldots, \log H_K(F_m)),
\]
such that  for each $1 \le i \le m$
 the polynomial $F_i(\underline{t}, \underline{Y})$ is irreducible over $K$.
\end{cor}
\begin{proof}
By the previous theorem, the number of $\underline{t} \in [B]_{\cO_K}^r$ for a given $B \ge 1$ for which $F_i(\underline{t}, \underline{Y})$ is reducible over $K$ is at most
\[
S_T(F_i, B) \ll_{K,n,r} 2^{O_{n,r}(d^n)}
 (\log H_K(F_i) +1)^{10}B^{r-1/2}\log(B)^{10(r-1)}.
\]
Recall that $\# [B]_{\cO_K} \ll_K B$.
Therefore, the number of $\underline{t}\in [B]_{\cO_K}^r$ for which all $F_i(\underline{t}, \underline{Y})$ are irreducible over $K$ is bounded from below  by
\[
\#[B]_{\cO_K}^r  - \sum_{i=1}^m S_T(F_i, B) \gg_{K,n,r} B^r - m B^{r-1/2} \log(B)^{10(r-1)}2^{O_{n,r}(d^n)}
\sum_{i=1}^m (\log H_K(F_i) +1)^{10}.
\]
In order to guarantee we have at least one
$\underline{t} \in [B]_{\cO_K}^r$, not lying on $X$, and such that  each specialization $F_i(\underline{t},\underline{Y})$ is
irreducible, it is enough to find $B$ satisfying the following inequality:
\begin{flalign}
\label{eq:Condition for irreducible specialization}
\tag{4.3}
\frac{B^{1/2}}{\log(B)^{10(r-1)}} &\gg_{K,n,r}  m 2^{O_{n,r}(d^n)}
\sum_{i=1}^m (\log H_K(F_i) +1)^{10}+ \frac{N_\mathrm{aff}(X,B)}{B^{r-1/2}\log(B)^{10(r-1)}}.
\end{flalign}
Using the Schwartz-Zippel bound, we have  $N_\mathrm{aff}(X,B) \ll_K s B^{r-1}$.
Note that $\frac{B^{1/2}}{\log(B)^{10(r-1)}}$ is an increasing function for $B \gg_n 1$ which is larger than a number $y$ as soon as $B=O_r(y^2 \log^{20(r-1)}(y))$.
Using Jensen's inequality on $g(y)=y^2 \log^{20(r-1)}(y)$ where $y$ is the RHS of (\ref{eq:Condition for irreducible specialization}), we conclude the desired quantity is positive, if we demand 
\begin{flalign*}
B \gg_{K,n,r} s^{2}\log(s)^{20(r-1)}+\left(m^2 2^{O_{n,r}(d^n)}
 \sum_{i=1}^m (\log H_K(F_i) +1)^{10} \right)^3. \qedhere
\end{flalign*}
\end{proof}
We now move to prove a version of Theorem \ref{thm:HIT.v2:Q} for global fields.
We need two well known lemmas that we record below.

\begin{lem}[{see e.g.~\cite[Theorem 11.2.7, Proposition 11.2.8]{Springer}}]
\label{lem:descent}
Let $K'/K$ be a Galois extension contained in $K^{\sepc}$, and let $Y$ be an affine $K'$-variety. Then $\sigma(Y) = Y$ for every $\sigma\in \Gal(K'/K)$ if and only if $Y$ can be defined over $K$.
\end{lem}

\begin{lem}
\label{lem:Bertini}
Let $n\ge 3$ and let $X \subset \PP^n$ be a geometrically irreducible hypersurface of degree $d$.
Furthermore set $C_K:=3$ if $\charac(K)=0$ and $C_K:=7$ if $\charac(K)>0$.
There exists a non-zero homogeneous form $ F
\in \ZZ[x_0,
\ldots,x_n]$ of degree at most $12(n+1)d^{C_K}$ such that if $V(\ell) \cap X$ is not geometrically irreducible for a linear form $\ell \in (\PP^{n})^*$, then
$F(\ell) =0$, where we reduce $F$ modulo $\charac(K)$ in the case of positive characteristic.
\end{lem}
\begin{proof}
We use Noether forms as in \Cref{thm:noether.forms}.
In characteristic $0$, we refer to \cite[Lemma 4.3.7]{CCDN-dgc} noting that the statement and proof in \cite{CCDN-dgc} should be amended by correctly concluding the degree of the form $F$ to be at most $(n-1)d(d^2-1)$, instead of the written expression $(n-1)(d^2-1)$. A similar correction should be made in \cite[Lemma 7.5]{Pared-Sas}. See \cite[Lemma 4.6]{Vermeulen:p} for $\charac(K)>0$. 
\end{proof}
We can now prove \Cref{thm:HIT.v2:Q}.

\begin{thm}
\label{thm:HIT.v2}
Let $n \ge 3$, and let $f \in \cO_K[x_0,\ldots,x_n]$ be an $r$-irreducible homogeneous degree $d$ polynomial for an integer $r \ge 1$.
Assume furthermore that all absolutely irreducible factors of $f$ can be defined over the separable closure of $K$.
 Then there exists a hypersurface $W \subset (\mathbb{P}^{n})^*$ defined over $K$ of degree at most $O_n(d^{C_K})$ where $C_K=3$ if $\mathrm{char}(K)=0$ and $C_K=7$ if  $\mathrm{char}(K)>0$, such that if $\ell \notin W$
 is a $K$-linear form,
 then $f|_{V(\ell)}$ is $r$-irreducible over $K$.
\end{thm}
\begin{remark}
In particular, if $f$ is homogeneous and  irreducible over $K$, we can deduce a version (up to a change of variables) of Hilbert's irreducibilty theorem for $n
\ge 4$ with $S_T(F,B)
\le O_{n}(d^{C_K})$.
\end{remark}
\begin{proof}
Factor $f=h_1 \cdot \ldots \cdot h_{m}$ into absolutely irreducible (homogeneous) factors and let $Z_i$ be the variety cut by $h_i$ in $\PP^n$. 
By Lemma \ref{lem:Bertini}, for each $i$ there exists a hypersurface $Y_i \subset (\PP^{n})^*$ of degree at most  $O_n(\deg(h_i)^{C_K})$
such that if $\ell \notin Y_i$, then $Z_i \cap H_\ell$
remains geometrically irreducible. Setting $Y=\bigcup Y_i$, then $Y$ is of degree at most  $O_n(d^{C_K})$.

Let $g$ be an irreducible component of $f$ over $K$ and write $g=h_{i_1} \cdot \ldots \cdot h_{i_s}$ and $I_g:=\{i_1,\ldots,i_s\}$.
Since $g$ is irreducible over $K$, and by our separability assumption, it follows that $\{h_{i}\}_{i \in I_g}$  are mutually co-prime, and for each $i,j \in I_g$ there exists $\sigma_{i,j} \in \mathrm{Gal}(\overline{K}/K)$ such that $\sigma_{i,j}(Z_i)=Z_j$ (see Lemma \ref{lem:descent}).
Set  $Z_{i,j}:= Z_i \cap Z_j$ and note that $\dim Z_{i,j} = n-2$.
Consider the variety  $Y’_{i,j}:= \{ \ell \in (\mathbb{P}^{n})^* : Z_{i,j}\subseteq H_{\ell}\}$ which  parameterizes hyperplanes $H_\ell$ containing $Z_{i,j}$.
Let $V_{i,j}=\mathrm{span}(Z_{i,j})$ be the  $\overline{K}$-vector space spanned by vectors in $Z_{i,j}$.
If $Z_{i,j} \subseteq H_\ell$, then $V_{i,j} \subseteq H_\ell$. It follows that for each $i,j\in I_g$, the variety $Y’_{i,j}$ can be either empty, a point, or a line according to the value of $\dim (V_{i,j})$.
Therefore, $Y_g':= \bigcup_{i,j \in I_g} Y’_{i,j}$  contains at most $\frac{\deg(g)(\deg(g)-1)}{2}$ irreducible components, which can be lines or points. If $\ell \not \in Y’_{i,j} \cup Y_i \cup Y_j$, then $\dim (Z_{i,j} \cap H_\ell)=n-3$, and therefore $h’_{i}|_{H_{\ell}}$ and  $h’_{j}|_{H_{\ell}}$ are co-prime. It follows that if $\ell \notin Y'_g \cup Y$ is a $K$-linear form, then $g|_{H_{\ell}}$ is  irreducible over $K$, as $\sigma_{i,j}(Z_{i}\cap H_\ell)=\sigma_{i,j}(Z_{i})\cap H_\ell=Z_{j}\cap H_\ell$ for every $i,j \in I_g$.

Write $f=g_1 \cdot \ldots \cdot g_{m'}$ with $d_i:=\deg(g_i)$ for a decomposition of $f$ into irreducible factors over $K$, and let $Y':=\bigcup Y'_{g_i}$. Then $Y'$ contains at most $ \sum \frac{d_i(d_i-1)}{2} \le \frac{d(d-1)}{2}$ irreducible components which are either points or lines.
We can thus find a hypersurface $W \supseteq Y \cup Y'$ of degree at most $O_n(d^{C_K})$ such that if $\ell \not \in W$ is a $K$-linear form,
then each $g_i|_{H_{\ell}}$ remains  irreducible over $K$, and therefore $f$ is $r$-irreducible.
\end{proof}
\subsection{The NCC condition}\label{sec:H.condition}

In this section we show that the condition of being NCC is preserved when cutting with a well-chosen hyperplane, and with most of its translates. This allows us to use induction in Section \ref{sec: Dimension growth for global fields}. In fact, using the theory of Hilbert schemes it is easy to prove such a non-effective result in characteristic zero, see e.g.\ \cite{Salb.upcoming}. However, proving an effective variant which also holds in positive characteristic is much more technical.

Fix an algebraic closure $\overline{K}$ of $K$, and a separable closure  $K^{\sepc} \subset \overline{K}$. Recall the condition of being NCC is defined as follows.

\begin{defn}\label{defn:H-cond.general}
Let $f\in K[x_1,...,x_n]$ for some $n\geq 3$. We say that $f$ is NCC over $K$ if there does not exist a $K$-linear map $\ell:\AA_K^n\to \AA_K^2$ and a curve $C$ in $\AA_K^2$ such that $V(f)=\ell^{-1}(C)$.
\end{defn}

Equivalently, $f\in K[x_1, \ldots, x_n]$ is NCC over $K$ if there do not exist $K$-linear forms $\ell_1(x), \ell_2(x)$ on $K^n$ and a polynomial $g\in K[y_1, y_2]$ such that $f(x) = g(\ell_1(x), \ell_2(x))$.

We denote by $(\PP^{n-1}_K)^*$ the dual space of $\PP^{n-1}_K$, which parametrizes hyperplanes in $\PP^{n-1}_K$, or in $\AA^{n}_K$ through the origin. We identify points of $(\PP^{n-1}_K)^*(K)$ and linear forms in $n$ variables over $K$.
If $S$ is a finite subset of $(\PP^{n-1}_K)^*(\overline{K})$, then we define
\[
H_S = \bigcup_{\ell\in S} V(\ell)\subset \AA_K^{n},
\]
which is the union of all hyperplanes in $\AA_K^{n}$ defined by a linear form in $S$ (we always take $H_S$ to be reduced).

We begin by restating the NCC condition, and show that being NCC over $K$ is equivalent to being NCC over $K^{\sepc}$.

\begin{notation}
Let $f \in K[x_1, \ldots, x_n]$ be a non-constant polynomial and let $f_i$ denote the homogenous degree $i$ part of $f$. Let $I:=\{ i: f_i \neq 0\}$. We set,
\[
W(f) := \bigcup_{i\in I} V(f_i)_{\redu} \subset \AA^{n}_K,
\]
where $V(f_i)_{\redu}$ denotes the reduced scheme cut by the radical of $(f_i)$ in $\AA_K^n$.
\end{notation}

Note that $W(f)$ is always a union of cones through the origin, since we consider $W(f)$ in affine space.

\begin{prop}\label{prop:equiv.H.cond}
Let $n\geq 3$, and let $f\in K[x_1, \ldots, x_n]$ be a non-constant polynomial.
Then the following conditions are equivalent:
\begin{enumerate}
\item $f$ is not NCC over $K$.
\item There exists a line $L\subset (\PP^{n-1}_K)^*$ defined over $K$ and a finite subset
$S \subset L(\overline{K})$
such that $W(f)= H_S$.
\end{enumerate}
\end{prop}

\begin{proof}
(1) $\Rightarrow$ (2): Suppose that $f$ is not NCC. Then there exist distinct elements $P = (a_1:\ldots: a_n), Q = (b_1:\ldots: b_n)$ in $(\PP^{n-1}_K)^*(K)$ corresponding to $K$-linear forms $\ell_1(x):=\sum_m a_m x_m$ and $\ell_2(x):=\sum_m b_mx_m$, and a polynomial $g$ in $K[y_1, y_2]$ such that
\[
f(x_1, \ldots, x_n) = g(\ell_1(x), \ell_2(x)).
\]
Let $L$ be the line in $(\PP^{n-1}_K)^*$ through $P$ and $Q$. Write $f = f_0 + f_1 + \ldots +f_d$ where $f_i$ is the homogeneous degree $i$ part of $f$, and similarly for $g_i$. Then $f_i(x) = g_i(\ell_1(x), \ell_2(x))$,
for every $i = 0, \ldots, d$. Since every
 homogeneous polynomial in two variables splits over $\overline{K}$, each $g_i$ splits over $\overline{K}$, and we have that if $f_i$ is non-constant then
\[
V(f_i)_{\redu} = V(g_i(\ell_1(x), \ell_2(x)))_{\redu} = H_{S_i},
\]
where $S_i$ is some finite subset of $L(\overline{K})$. Hence $W(f) = H_{\cup_i S_i}$.

(2) $\Rightarrow$ (1): For the converse, suppose that (2) holds with  $P = (a_1 : \ldots : a_n), Q = (b_1: \ldots: b_n) \in (\PP_K^{n-1})^*(K)$ on the line $L$, and $S \subset L(\overline{K})$ a finite set. Let $\ell_1(x):=\sum_m a_m x_m$ and $\ell_2(x):=\sum_m b_mx_m$ denote the corresponding $K$-linear forms.

Writing $f=\sum\limits_{i=0}^d f_i$ where $f_i$ is homogeneous of degree $i$ as before, it is enough to show that for each $i$, we have $f_i=g_i(\ell_1(x), \ell_2(x))$ where $g_i\in K[y_1,y_2]$.
Fix $i$ such that $f_i$ is non-constant. Each $\ell \in S$ can be written as a $\overline{K}$-linear combination of $\ell_1, \ell_2$. Therefore, since
$
V(f_i)_{\redu} \subseteq W(f)
=H_S=\bigcup\limits_{\ell\in S} V(\ell),
$ there exists a polynomial
$h_i \in \overline{K}[y_1,y_2]$ such that $f_i(x)=h_i(\ell_1(x), \ell_2(x))$.
Write
\[
f_i(x)=h_i\left( \ell_1(x), \ell_2(x)\right) = \sum_j \lambda_j h_{ij}\left( \ell_1(x), \ell_2(x)\right),
\]
where $\{\lambda_j\}_j$ forms a $K$-basis for $\overline{K}$, $\lambda_1=1$, and $h_{ij} \in K[y_1,y_2]$.
Since $f_i$ is defined over $K$, and  $h_{ij}\left( \ell_1(x), \ell_2(x)\right)\in K[x_1,\ldots, x_n]$, we get by the linear independence of $\{\lambda_j\}_j$ that $h_{ij}=0$ unless $j=1$.
Therefore $h_i(y_1,y_2)=h_{i1}(y_1,y_2)$ is defined over $K$ as required.
\end{proof}
We now move to showing that a hypersurface is NCC over $K$ if and only if it is NCC over the separable closure $K^{\sepc}$. We first prove a proposition which  characterizes the dimension of the smallest $K$-linear space containing a given $\overline{K}$-point $P$ in $\mathbb{P}^m_K(\overline{K})$.
\begin{prop}
\label{prop:characterization of delta(P)}
Let $P=(a_1:\ldots:a_{m+1}) \in \mathbb{P}^m(\overline{K})$ be a point, and define
$\delta_K(P):=\dim\mathrm{span}_{K}\{ a_1,\ldots ,a_{m+1}\}$. Then $P$ lies on a unique $K$-linear space of dimension $\delta_K(P)-1$, and is not contained in any $K$-linear space of smaller dimension.
\end{prop}

\begin{proof}
First note that $\delta_K(P)$ is well defined, since for any non-zero $\lambda \in \overline{K}$ the collection $\{a_{i_j}\}_j$ is linearly independent if and only if $\{\lambda a_{i_j}\}_j$ is linearly independent.
For $\delta_K(P)=1$ the statement is clear; $\delta_K(P)=1$ if and only if $P$ is defined over $K$. Let $s>1$ and assume by induction that the statement holds for every  point $P'$ with $\delta_K(P')<s$.

If $\delta_K(P)=s$, we may assume without loss of generality that  $\{ a_1, \ldots,a_s\}$ forms a $K$-basis for $\mathrm{span}_K\{ a_1,
\ldots , a_{m+1}\}$, and write $a_j=\sum_{i=1}^{s}a_i b_{ij}$ for each $j$, and some $b_{ij}\in K$. We get $K$-points $Q_i=(b_{i1}:\ldots:b_{i\,m+1})\in \mathbb{P}^m(K)$, where $1 \le i \le s$.
Note that  $b_{ij}=1$ if $i=j$ and $b_{ij}=0$ if $i \neq j$ and $j \le s$. Therefore $\{Q_i\}_{i=1}^s$ are linearly independent over $K$, and the $K$-linear space $L\subseteq \PP^{m}$ spanned by $Q_1,\ldots,Q_s$ has dimension $s-1$, and contains $P$.
If $P$ is contained in a different $K$-linear space $L'$ of dimension $s-1$, then we have linearly independent points $Q'_i=(b'_{i1}:\ldots:b'_{i\,m+1})\in L'(K)$ and we may write $a_j=\sum_{i=1}^{s}\lambda_i b'_{ij}$
for each $j$ and some $\lambda_i \in \overline{K}$. Since $\{ a_1,\ldots,a_s\} \subset \mathrm{span}_K\{ \lambda_1,\ldots,\lambda_s\}$ is a $K$-basis, we may write $\lambda_j= \sum_{k=1}^s a_k\alpha_{jk}$ where $\alpha_{jk}\in K$. Rewriting, we get for each $j$,
\[
\sum_{i=1}^{s}a_i b_{ij}
=a_j
=\sum_{i=1}^{s}\lambda_i b'_{ij}
=a_1\left(\sum_{k=1}^s \alpha_{k1}b'_{kj}\right)
+\ldots+a_s\left(\sum_{k=1}^s \alpha_{ks}b'_{kj}\right).
\]
Since $\{ a_1,\ldots,a_s\}$ is a $K$-basis, we get
$b_{ij}=\sum_{k=1}^s \alpha_{ki}b'_{kj}$
 for all $i,j,k$, implying that $\{Q_1 ,\ldots,Q_s\} \subset L'$. It  follows that $L=L'$ is unique, and that $P$ is not contained in any $K$-linear space of dimension smaller than $s-1$, as that would contradict the uniqueness of $L$.

If $P$ is contained in a unique $K$-linear space $L$ of dimension $s-1$, then $\delta_K(P) \ge s$ since otherwise there exists by induction  a $K$-linear space of dimension smaller than $s-1$ containing $P$, and therefore there are infinitely many $K$-linear spaces of dimension $s-1$ on which $P$ lies. Since $L$ is defined over $K$,
there exist points $Q_1, \ldots ,Q_s\in L(K)$, with $Q_i=(b_{i1}:\ldots:b_{i\,m+1})$ where we may assume $b_{ij} \in K$ for all $i,j $, and scalars $\lambda_1, \ldots ,\lambda_s \in \overline{K}$ such that
\[
P=\lambda_1 Q_1 + \ldots +\lambda_s Q_s
.
\]
 We get  $a_j=\sum_{i=1}^{s}\lambda_i b_{ij}$ for each $j$, implying $\{ a_1,\ldots,a_{m+1}\} \subset \mathrm{span}_K\{\lambda_1, \ldots,\lambda_s\}$,
and therefore $\delta_K(P) \le s$, implying $\delta_K(P)=s$.
\end{proof}

\begin{prop}
\label{prop:H over k is equivalent to H over Ksep}
Let $n\geq 3$, and let $f\in K[x_1, \ldots, x_n]$ be a non-constant polynomial. Then $f$ is NCC over $K$ if and only if $f$ is NCC over $K^{\sepc}$.
\end{prop}

\begin{proof}
If $f$ is not NCC over $K$, then clearly it is not NCC over $K^{\sepc}$.

If $f$ is not NCC over $K^{\sepc}$, by Proposition \ref{prop:equiv.H.cond} there exists a line $L \subset (\PP^{n-1})^*$ defined over $K^{\sepc}$
 and a finite set $S \subset L(\overline{K})$
such that $W(f)= H_S$.
Since $f$ is defined over $K$, we have that
$\sigma(W(f)) = W(f)$ for all $\sigma\in \Aut(\overline{K}/K)$ and therefore $\sigma(S) = S$ for all $\sigma\in \Aut(\overline{K}/K)$.

If $|S|\ge 2$, since $S$ consists of points on the line $L$, then the automorphism group $\Aut(\overline{K}/K)$ must preserve $L$.
 Since $L$ is defined over $K^{\sepc}$, this implies that $\sigma(L) = L$ for every $\sigma\in \Gal(K^{\sepc}/K)$. Hence \Cref{lem:descent} implies that $L$ can be defined over $K$, which proves the required statement using Proposition \ref{prop:equiv.H.cond}.

We may thus assume that $S = \{\ell \}$ is a singleton.  Then we have that $\sigma(\ell) = \ell$ for every $\sigma \in \Aut(\overline{K}/K)$. Write $\ell = (c_1 : \ldots : c_n)$, then without loss of generality we may assume that $c_1 = 1$. It follows that $\sigma(c_i) = c_i$ for every $\sigma\in \Aut(\overline{K}/K)$.

Consider now $\delta_{K^{\sepc}}(\ell)$ as in Proposition \ref{prop:characterization of delta(P)}. Since $\ell$ lies on the $K^{\sepc}$-line  $L$, then $\delta_{K^{\sepc}}(\ell) \le 2$. If $\delta_{K^{\sepc}}(\ell)=1$, then $\ell$ is defined over $K^{\sepc}$, and therefore $\ell$ is defined over $K$, since every element of $\Gal(K^{\sepc}/K)$ fixes $\ell$. We may thus take any line defined over $K$ which  passes through $\ell$ to deduce Condition (2) of Proposition \ref{prop:equiv.H.cond}.

Assume $\delta_{K^{\sepc}}(\ell)=2$. By Proposition \ref{prop:characterization of delta(P)}, there exists a unique line $L$ defined over $K^{\sepc}$ which contains $\ell$. Since for each  $\sigma\in \Aut(\overline{K}/K)$ we have
 $\sigma(\ell)=\ell$, the line
$\sigma(L)$ is defined over $K^{\sepc}$ and contains $\ell$. We get $\sigma(L)=L$ for every
 $\sigma\in \Aut(\overline{K}/K)$, and thus $L$ is defined over $K$ by Lemma \ref{lem:descent}.
%
%
\end{proof}

In characteristic zero
 the condition of being NCC may be reformulated  using the partial derivatives of $f$. This allows us to simplify the proof of  \Cref{prop:preservation.H.cond} in that case.

\begin{lem}\label{lem:H.cond.partial.derivatives}
Assume that $K$ has characteristic zero. Let $f\in K[x_1, \ldots, x_n]$ be non-constant, with $n\geq 3$. Then $f$ is NCC over $K$ if and only if the $K$-vector subspace $V$ of $K[x_1, \ldots, x_n]$ spanned by $\left\{\frac{\partial f}{\partial x_i}\right\}_{i=1}^n$ has dimension at least $3$.
\end{lem}

\begin{proof}
If $f$ is not NCC, then clearly $\dim_K V \leq 2$.

Conversely, assume that $\dim_K V \leq 2$. Without loss of generality, we can assume that $\frac{\partial f}{\partial x_1}, \frac{\partial f}{\partial x_2}$ span $V$. For each $i$ there are $a_i, b_i$ in $K$ such that
\[
\frac{\partial f}{\partial x_i} = a_i \frac{\partial f}{\partial x_1} + b_i \frac{\partial f}{\partial x_2}.
\]
Consider the polynomial
\[
g(y_1, \ldots, y_n) = f\left( y_1 - \sum_{i\geq 3} a_i y_i, y_2 - \sum_{i\geq 3} b_i y_i, y_3, \ldots, y_n\right) \in K[y_1, \ldots, y_n].
\]
Then $\frac{\partial g}{\partial y_i} = 0$ whenever $i\geq 3$, and so $g(y)$ is a polynomial only involving $y_1, y_2$. Hence $f$ is not NCC.
\end{proof}
%

We can now prove that being NCC is preserved when intersecting with suitably chosen hyperplanes.
\begin{prop}\label{prop:preservation.H.cond}
Let $K$ be a global field, fix $n\geq 4$, and
let $f\in K[x_1, \ldots, x_n]$ be a polynomial of degree $d\geq 2$. Assume that $f$ is NCC over $K$.
\begin{enumerate}
\item
Set $C_K=3$ if $\mathrm{char} K =0$ and  $C_K=7$ if $\mathrm{char} K >0$.
There exists a hypersurface $W \subset (\PP^{n-1}_K)^*$ of degree at most $O_n(d^{C_K})$ such that
if $\ell \notin W$,
\begin{enumerate}
\item $f|_{V(\ell)}$ is NCC over $K$; and
\item $f_d|_{V(\ell)}$ is not identically zero.
\end{enumerate}
In particular, there exists a non-zero $K$-linear form $\ell$ of height at most  $O_n(d^{C_K})$ satisfying Conditions (a) and (b) above.
\item Set $C_K=1$ if $\charac K = 0$ and $C_K = 7$ if $\charac K > 0$. Let $\ell_1, \ell_2$ be $K$-linear forms. Assume that $f|_{V(\ell_1)}$ is NCC over $K$ and that $f_d|_{V(\ell_1)}$ is  not identically zero. There are at most $O_n(d^{C_K})$ values $a\in K$ such that one of the two following conditions hold
\begin{enumerate}
\item $f|_{V(\ell_1 + a\ell_2)}$ is not NCC over $K$; or
\item $f_d|_{V(\ell_1 + a\ell_2)}$ is identically zero.
\end{enumerate}
\item Set $C_K = 1$ if $\charac K = 0$ and $C_K = 2$ if $\charac K > 0$. Let $\ell$ be a $K$-linear form. Assume that $f|_{V(\ell)}$ is NCC over $K$ and that $f_d|_{V(\ell)}$ is not identically zero. If $\charac K > 0$ assume moreover that $f_d|_{V(\ell)}$ is not a power of a $K$-linear form, up to a constant. Then there are at most $O(d^{C_K})$ elements $b\in K$ such that $f|_{V(\ell - b)}$ is not NCC over $K$.
\end{enumerate}
\end{prop}

\begin{proof}
First note that $f_d|_{V(\ell)} = 0$ if and only if $\ell$ is a factor of $f_d$. There are thus at most $d$ linear forms $\ell$ such that $f_d|_{V(\ell)} = 0$, which we avoid in the proof below.\\

(1). We use \Cref{prop:equiv.H.cond}.
If $W(f)$ contains a geometrically  irreducible component $Z$ of degree at least $2$, by Lemma \ref{lem:Bertini}, there exists a form $F \in \ZZ[x_1,\ldots,x_n]$  in the coefficients of $\ell$ of degree at most $O_n(d^{C_K})$, such that if $F(\ell) \neq 0$, then $Z \cap V(\ell)$ is geometrically irreducible, and therefore  $f|_{V(\ell)}$ is NCC over $K$.

Assume now that $W(f)$ contains at least three hyperplanes whose corresponding linear forms are linearly independent, and denote the linear space they span in $(\PP^{n-1}_K)^*$ by $V$. Then $W(f|_{V(\ell)})$ contains three such hyperplanes for any $K$-linear form $\ell$ which  does not lie on $V$,  and for which $f_i|_{V(\ell)}\neq 0$ for indices  $i$ where $W(f_i)$ contains one of the hyperplanes as above. It follows by \Cref{prop:equiv.H.cond} that $f|_{V(\ell)}$ is NCC over $K$ in this case.
We may thus find a hypersurface
of degree at most $O(d)$
in $(\PP^{n-1}_K)^*$ such that $f|_{V(\ell)}$ is NCC over $K$ for every $\ell$ not lying on it.


Finally, assume that $W(f)=H_S$, where $S$ is a finite set contained in a $\overline{K}$-line $L \subseteq (\PP^{n-1}_K)^*$ which is not defined over $K$.
By Proposition \ref{prop:H over k is equivalent to H over Ksep}, we may assume $K$ is separably closed and therefore if $\charac(K)=0$ we are done by  \Cref{prop:equiv.H.cond}.
It follows by Proposition \ref{prop:characterization of delta(P)} that
there exists a hyperplane $V(\ell') \subset W(f)$, corresponding to a point $\ell'=(a_1:\ldots:a_n)\in (\PP^{n-1}_{K})^*$, satisfying $\delta_K(\ell')\ge 2$. If $\delta_K(\ell')\ge 3$, then without loss of generality we may assume that $a_1, a_2, a_3$ are linearly independent over $K$. We may thus choose a $K$-linear form $\ell(x)=\sum b_i x_i$ of height at most $O(d)$ such that $\delta_K(\ell'|_{V(\ell)})\geq 3$ by taking $b_i \neq 0$ for any $i \geq 4$. In particular, any $\ell$ not lying on the hypersurface cut in $(\PP^{n-1}_{K})^*$ by $\prod b_i =0$ is as desired.
We then get that $ V(\ell')\cap{V(\ell)} \subset W(f|_{V(\ell)})$, considered as a point in $(\PP_K^{n-2})^*$, does not lie on any line defined over $K$ by Proposition  \ref{prop:characterization of delta(P)}, and therefore $f|_{V(\ell)}$ is NCC.
We may thus assume that $\delta_K(\ell') \le 2$ for every hyperplane $V(\ell') \subset W(f)$.
Fix hyperplanes $H_{\ell_1}, H_{\ell_2} \subset W(f)$ with $\delta_K(\ell_1)=2$ and such that $\ell_2$ does not lie on the unique $K$-line  $L'\subset (\PP^{n-1}_K)^*$ containing $\ell_1$.
Note that  $ \ell_2 \notin L\subset (\PP^{n-1}_K)^*$ if and only if $V_{L'}:=\{\ell_1'(x)=\ell_2'(x)=0\} \not \subset H_{\ell_2}$, where $\ell_1'$, $ \ell_2'$ are two linearly independent $K$-linear forms lying on $L'$. We may thus choose any $\ell$ such that the conditions  $H_{\ell_2} \cap V_{L'} \not \subset V(\ell)$,
$\delta_K({\ell_1}|_{V(\ell)})=2$
and $f_d|_{V(\ell)} \neq 0$ hold.
 Therefore, noting that $\dim (H_{\ell_2} \cap V_{L'})=n-3$ and using arguments similar to the $\delta_K \ge 3$ case, we may find a hypersurface in $(\PP^{n-1}_K)^*$ of degree at most $O_n(d)$ such that every linear form not lying on it satisfies our desired condition. The final claim follows from the Schwartz--Zippel bound. \\


(2). If $W(f)$ contains a non-linear geometrically irreducible component or at least three hyperplanes whose corresponding linear forms are linearly independent, we may
follow the proof of (1) with the linear forms $\tilde{\ell}_a(x):=\ell_1(x) + a\ell_2(x), a\in K$, noting that $f|_{V(\tilde{\ell}_0)}$ is NCC.
We may then bound the number of values of $a$  for which $f|_{V(\tilde{\ell}_a)}$ is not NCC by $O_n(d^{C_K'})$ with $C_K'$ as in (1).
If we are not in one of the above two situations, then $W(f)=H_S$ where $S \subset L$ is a finite set contained in a line $L \subseteq (\PP^{n-1}_{K})^*$ not defined over $K$. In particular, each homogeneous part $f_i$ of $f$ splits into linear factors.
 Assume towards a contradiction that  $f|_{V(\ell_1+a_i\ell_2)}$ is not NCC for distinct values $a_1 \neq a_2$ in $K$. Then there exist $K$-lines $L_i \subset (\mathbb{P}^{n-2}_K)^*$ such that $W(f|_{V(\ell_1+a_i\ell_2)})=H_{S_i}$ and ${S_i} \subset L_i$ for $i=1,2$. It follows that $S \subset P_i$, where $P_i$ is the $K$-plane in $(\PP^{n-1}_{K})^*$ containing the preimage of $L_i$ under the projection 
 \[
 \mathrm{pr}_i:(\PP^{n-1}_{K})^*\setminus \{ \ell_1+a_i\ell_2\} \to (\PP^{n-2}_{K})^*: \ell \mapsto \ell_{|_{V(\ell_1+a_i\ell_2)}}.
 \]
Note that we must have  $P_1 \neq P_2$, as otherwise we get that $\ell_1$ lies in $P_1$,
implying that $P_1$ gives rise to a $K$-line in $(\PP^{n-2}_{K})^*$ containing $W(f|_{V(\ell_1)})$, but $f|_{V(\ell_1)}$ is NCC.
We thus get $S \subset P_1 \cap P_2$, where $P_1 \cap P_2$ is either a $K$-line or a $K$-point. Since $f$ is NCC, it follows there cannot be such distinct values $a_1, a_2$.

To improve this bound when $\charac (K) = 0$, we use \Cref{lem:H.cond.partial.derivatives} as follows.
Write $\ell_1(x)=\sum a_i x_i$, and $\ell_2(x)=\sum b_i x_i$, and without loss of generality assume $a_1 \neq 0$.
Substituting 
\[
x_1:=-\frac{1}{a_1+a b_1}\sum_{i=2}^{n} (a_i + a b_i) x_i
\]
into $f$, we get a family of polynomials $\tilde{f}(x_2,\ldots,x_n,a)$ in $n-1$ variables depending on the parameter $a \in K$ (if $b_1 \neq 0$, we assume $a_1 +a b_1 \neq 0$). Let $V_{\le d-1}$ be the vector space of polynomials of degree at most $d-1$ taken with the standard basis of monomials and set $r:=\dim V_{\le d-1}$. Consider the $r\times (n-1)$ matrix $M(a)$ whose columns are the partial derivatives of $\tilde{f}$ written in terms of this basis. By \Cref{lem:H.cond.partial.derivatives}, for a given $a\in K$, the polynomial $\tilde{f}(x,a)$ is NCC if and only if $M(a)$ has rank $3$, or equivalently $M(a)$ has a non-vanishing $3\times 3$ minor. By assumption, $M(a)|_{a=0}$ has a non-vanishing $3\times 3$ minor $M'(0)$. Since after clearing denominators $M'(a)=0$ is a polynomial equation in $a$ of degree at most $6d$, there can be at most $O(d)$ values of $a$ for which $M(a)$ has rank$\le2$. This proves the claim with the required bound.\\

(3). If $\charac K = 0$, one can consider the family $\tilde{f}(x,b)$ obtained by substituting $\ell(x)=b$ into $f$, and reason as in (2) using \Cref{lem:H.cond.partial.derivatives}. We may therefore assume that $\charac K = p > 0$. Using \Cref{prop:equiv.H.cond}, we may further assume that $K$ is separably closed.

Write $g_b(y):=f|_{V(\ell-b)}$ for the polynomial in new variables $y = (y_1, \ldots, y_{n-1})$ and coefficients in $K[b]$ obtained by substituting the equation $\ell(x)=b$ into $f$.
Let $g_b(y) = \sum_i g_{ib}(y)\in K[b][y]$, where $g_{ib}(y)$ is homogeneous of degree $i$. Note that the homogeneous part $g_{db}(y)$ of top degree $d$ is independent of $b$, so we may write $g_{d0}(y)$ instead of $g_{db}(y)$.
If $g_{d0}$ is NCC, then it is clear that $g_b(y)$ is NCC for every $b\in K$. We can thus assume that $g_{d0}$ is not NCC, and therefore by  \Cref{prop:equiv.H.cond}
 there exists a line $L \subset (\PP^{n-2}_K)^*$ defined over $K$ and a finite set $S \subset L(\overline{K})$ such that $W(g_{d0})=H_S$.
We claim that this line $L$ is unique.
If $S$ contains at least two distinct points, they uniquely determine the line $L$. Otherwise, $S=\{\ell_1\}$ is a singleton, and by our assumption that  $g_{d0}=f_d |_{V(\ell)}$ is not a power of a $K$-linear form, $\ell_1$ is not defined over $K$. By Proposition \ref{prop:characterization of delta(P)}, since $g_{d0}$ is not NCC, it follows that $\delta_K(\ell_1)=2$, implying that $L$ is unique.
It now follows by the above that if $g_{b}(y)$ is not NCC for some $b \in K$, then there exists a finite set $S_{b} \subset L(\overline{K})$ such that $W(g_{b})=H_{S_{b}}$. Crucially, for every value $b \in K$ for which $g_{b}$ is not NCC, we must always use the same $K$-line $L$ in \Cref{prop:equiv.H.cond}.

Since $g_0(y)$ is NCC, $W(g_0)$ is not a finite union of hyperplanes corresponding to points in $L(\overline{K})$. It follows that there exists an $i$ such that $g_{i0}(y)\neq 0$ and $g_{ib}(y)$ has a $K(b)$-irreducible factor $h_b(y)$ which is primitive over $K[b]$, such that $W(h_b|_{b=0})\neq H_{S}$ for any finite set $S \subset L(\overline{K})$.
We note that there can be at most $d$ values of $b$ for which $g_{ib}(y)$ is the zero polynomial, and that these values do not add to our final bound. Consider the closed incidence variety $\Gamma$:
\[
\Gamma = \{ (b,\lambda) \in \mathbb{A}^1_K \times L : V(\lambda) \subseteq V(h_b)\} \subset \mathbb{A}^1_K \times L.
\]
We have $\dim \Gamma \le 1$, and we may write $\Gamma =C_0 \cup C_1$ where each $C_i$ is equidimensional of dimension $i$, and $C_1 \cap C_0 = \varnothing$.
Since $h_b$ is $K[b]$-primitive, $h_{b}$ is not identically zero for any $b \in K$, and thus $\Gamma$ cannot contain lines of the form $\{b_0\} \times L$.
Let $u_k(b,\lambda)$ be the defining equations for $\Gamma$ in $\AA^1_K \times L$ obtained by demanding that the restriction of $h_b(y)$ to the hyperplane $V(\ell_\lambda) \subset \AA^{n-1}_K$ is the zero polynomial. Note that each $u_k(b,\lambda)$ is of degree at most $d$ in both variables, and that since $C_1$ is a hypersurface in $\mathbb{A}^1_K \times L$, it is defined by a single equation $u(b,\lambda)$ which must be a factor of each $u_k$. We conclude that $C_0$ is cut in $\AA^1_K \times L$ by the equations $u_k/u$, and therefore it is of size at most $2d^2$.

Let $\mathrm{proj}:\AA^1_K \times L \to \AA^1_K$ denote the projection map. We claim now that the only values of $b$ for which $W(h_b)$ may be contained in a finite union of hyperplanes parameterized by points in $L$, and therefore $g_b$ may not be NCC, are values of $b$ which lie in $\mathrm{proj}(C_0)$.
Indeed, set $e = \deg_y(h_b(y))$ and consider the degree of the map $\mathrm{proj}|_{C_1}: C_1\to \AA^1_K$. It cannot be larger than $e-1$, since $h_b|_{b=0}(y)$ does not decompose as a product of linear factors parametrized by points in $L$, and thus the fiber of $\mathrm{proj}|_{C_1}$ over $0$ cannot be larger than $e-1$, where we also count multiplicities. It follows the fiber over any $b \in \AA^1_K$ has size at most $e-1$, counting multiplicities.
We conclude there can be at most $\#C_0\le 2d^2$ possible values for which $g_b=f|_{V(\ell-b)}$ is not NCC, completing the proof.
\end{proof}
Note that we cannot remove the hypothesis that $f_d |_{V(\ell)}$ is not a power of a linear form in \Cref{prop:preservation.H.cond}(3) if $\charac K > 0$, as the next example shows.
\begin{example}
Let $K = \FF_p(t)$, and take $m\geq 1$ and $d > p^m + 1$. Define
\[
f(x, y, z, w) = (y+z+w)^d + x(y^{p^m}+ z^{p^m} + (t^{p^m}+1)w^{p^m})+y^{p^m}+ tz^{p^m}+ (t^{p^m}+ t)z^{p^m}.
\]
Then we have that
\begin{align*}
W(f) =& \{y+z+w = 0\}\cup \{x=0\}\cup \{y+ z+(t+1)w = 0\}\\
& \cup \{y+t^{1/p^m}z + (t+t^{1/p^m})w = 0\}.
\end{align*}
Hence by \Cref{prop:equiv.H.cond}, $f$ is NCC. Consider the linear form $\ell = x$. Then
\[
W(f|_{V(x)}) = \{y+z+w = 0\} \cup \{y+t^{1/p^m}z + (t+t^{1/p^m})w = 0\},
\]
so that $f|_{V(x)}$ is still NCC. However, if $b\in \FF_p(t)$ then
\begin{align*}
W(f|_{V(x-b)}) = \{y+z+w=0\}\cup \left\{y + \left(\frac{b+t}{b+1}\right)^{1/p^m}z + \left(t + \left(\frac{b+t}{b+1}\right)^{1/p^m}\right)w = 0\right\},
\end{align*}
and if $b=\frac{t - u^{p^m}}{u^{p^m}-1}$ for any $1 \neq u\in \FF_p(t)$, then $f|_{V(x-b)}$ is not NCC.
\end{example}

\subsection{Dimension growth}
\label{sec: Dimension growth for global fields}

Throughout this section we fix a global field $K$ with ring of integers $\cO_K$. The goal of this section is to outline a proof of \Cref{thm:fd:K}, following along the lines of \Cref{sec:proof}. We will however be slightly less precise in optimizing all constants. The base case for induction is again the case of affine surfaces. The key here is counting on lines on the surface.

\begin{remark}\label{rem:Q:K}
In the proofs of \Cref{prop:1} and \Cref{thm:fd:K}, we will use some results of the previous sections that adapt naturally mutatis mutandis to the case of global fields,  but which we leave to the reader to make explicit. It concerns e.g.\,\Cref{prop:aux.affine.poly} which generalizes to \cite[Theorem 5.14]{Pared-Sas}, and \Cref{prop:pila.hypersurface,prop:pila} in which we may need to use Noether forms in positive characteristic (see \Cref{thm:noether.forms}).
\end{remark}

\begin{prop}\label{prop:1}
Let $C_K = 7$ if $\charac K = 0$ and $C_K = 11$ if $\charac K > 0$. There exist constants $c$ and $\kappa$ depending on $K$ such that we have the following for all polynomials $f$ in $\cO_K[y_1,y_2,y_3]$ of degree $d\geq 3$ whose homogeneous part of highest degree $f_d$ is $2$-irreducible over $K$, and such that $f$ is both irreducible and NCC over $K$; For every integer $B \ge 2$,
$$
N_{\rm aff}(f,B)\leq c d^{C_K} B, \quad \mbox{ when $d\ge 5$},
$$
$$
N_{\rm aff}(f,B)\leq c (\log B)^\kappa B, \quad  \mbox{ when $d=4$},
$$
and,
$$
N_{\rm aff}(f,B)\leq c (\log B)^\kappa  B^{2/\sqrt{3}} \quad  \mbox{ when $d=3$}.
$$
If $f_d$ is only $1$-irreducible over $K$, we  have that
\begin{align*}
N_{\rm aff}(f, B) \leq c d^{C_K} B (\log B)^\kappa, \quad \mbox{ when $d\geq 4$}, \\
N_{\rm aff}(f,B)\leq c (\log B)^\kappa  B^{2/\sqrt{3}} \quad  \mbox{ when $d=3$}.
\end{align*}
\end{prop}

\begin{proof}
The strategy is the same as in Proposition \ref{prop:0.4strong}. First, one counts $\cO_K$-points which are contained in lines on the surface $X=V(f)$ as in \Cref{prop1:Brow-Heath-Salb}. There is a slight subtlety between being NCC over $K$ and over an algebraic closure $\overline{K}$. If $X$ is NCC over $\overline{K}$ then we may reason exactly as in the proof of \Cref{prop1:Brow-Heath-Salb}, by noting that not all $c_i(a,v)$ as in \Cref{prop1:Brow-Heath-Salb} vanish. If $X$ is not NCC over $\overline{K}$ but it is NCC over $K$, then this means that $X$ contains infinitely many lines whose direction $v$ is not defined over $K$. In particular, any such line has at most one $K$-point, and \Cref{prop1:Brow-Heath-Salb} still holds (since counting on these lines may be absorbed in $\#I$).
One obtains that if $I$ is a finite set of lines on $X$, then
\[
N_{\rm aff}(X\cap (\cup_{L\in I}L), B) \ll d^{4} B + \#I.
\]
Note  that for the proof of this result, one should replace the use of~\cite[Theorem 2]{BCN-d} for counting on projective curves from by the analogous result from \cite{BinCluKat}.  

Once we have the above result for counting integral points on lines on $X$, one proceeds in a similar fashion to obtain an auxiliary polynomial $g$ not divisible by $f$ and vanishing on all integral points on $X$ of height at most $B$. Such a $g$ exists by~\cite[Thm.\,5.14]{Pared-Sas}. Following the proof of \Cref{prop:0.4strong} then gives the result.
\end{proof}

We now sketch the proof of dimension growth, by cutting with hyperplanes while preserving NCC and $2$-irreducibility via our effective Hilbert irreducibility results. We first prove a general version of \Cref{cor:preserve.H.cond.with.HIT.Q}.
\begin{cor}\label{cor:preserve.H.cond.with.HIT}
Let $f \in \cO_K[x_1, \ldots, x_n]$ be a polynomial  of degree  $d\geq 2$ with $n\geq 4$. Assume that $f$ is NCC and that $f_d$ is  $r$-irreducible over $K$ for some integer $r \ge 1$.
Let $C_K:=3$ if $\charac(K)=0$, and $C_K:=7$ otherwise.
Then there exist linearly independent linear forms $\ell, \ell'$ of height at most $O_n(d^{C_K})$ and  $ t\in [O_n(d^{C_K})]_{\cO_K}$
 such that $f_d|_{V(\ell- t\ell')}$ is $r$-irreducible, and such that there are at most $O_n(d)$ values $b\in K$ for which $f|_{V(\ell- t\ell'- b)}$ is not NCC.
\end{cor}
\begin{proof}
By combining \Cref{prop:preservation.H.cond}(1) and \Cref{thm:HIT.v2}, we may find a hypersurface $W'\subset (\PP^{n-1}_{K})^*$ of degree at most $O_n(d^{C_K})$, such that if $\ell$ is a linear form avoiding $W'$, then $f|_{V(\ell)}$ is NCC and $f_d|_{V(\ell)}$ is $r$-irreducible and not identically zero.
 Using the Schwartz--Zippel bounds, we may find a $K$-linear form $\ell$ of height at most $O_n(d^{C_K})$ which does not lie on $W'$. Since $\ell \notin W'$, for every  linear form $\ell'$ linearly independent of $\ell$, the intersection of the line passing through $\ell$ and $\ell'$ with $W'$ has at most $O_n(d^{C_K})$ points. It follows  there are at most $O_n(d^{C_K})$  values $t \in K$ for which $f_d|_{V(\ell+t\ell')}$ is not $r$-irreducible.
The claim now follows using
 \Cref{prop:preservation.H.cond}(2),(3).
\end{proof}

\begin{proof}[Proof of \Cref{thm:fd:K}]
Recall that we are given an irreducible $f\in \cO_K[x_1, \ldots, x_n]$ of degree $d\geq 3$ whose highest degree part is $2$-irreducible (or $1$-irreducible) over $K$ and we wish to bound the number of $\cO_K$-points on $X = V(f)$ of height at most $B$. As usual, we can assume that $f$ is absolutely irreducible. We follow the same proof strategy as for \Cref{thm:fd:Q}.

Let us first assume that $d> \kappa$, for some constant $\kappa = O_{K,n}(1)$. We use~\cite[Thm.\,5.14]{Pared-Sas} to find an auxiliary polynomial $g$ of degree at most
\[
\ll_{K,n} d^{e_1} B^{\frac{1}{d^{n-1}}}
\]
which vanishes on all integral points of $X$ of height at most $B$, where $e_1 = 4$ in characteristic $0$ and $e_1 = 7$ in positive characteristic. We then count integral points on all irreducible components of $X\cap V(g)$ as in the proof of \Cref{thm:fd:Q}. The components of degree $1$ contribute at most $O_{K,n}(d^{e}B^{n-2})$ if $f_d$ is $2$-irreducible, and at most $O_{K,n}(d^e B^{n-2}\log(B)^{O_K,n(1)})$ if $f_d$ is $1$-irreducible, where $e = 7$ in characteristic $0$ and $e = 11$ in positive characteristic. This larger exponent for positive characteristic comes from the fact that in \Cref{lem:number.of.n-2.planes} one uses Noether forms, which are of degree $O(d^6)$ in positive characteristic~\cite{KALTOFEN1995} rather than of degree $d^2$~\cite{RuppertCrelle}, see \Cref{thm:noether.forms}. For the components of higher degree, one reasons exactly as in the proof of ~\Cref{thm:fd:Q}.


Assume now that $d \leq \kappa = O_{K,n}(1)$. If $f_d$ is $1$-irreducible, or $f_d$ is $2$-irreducible and its absolutely irreducible factors can be defined over $K^{\mathrm{sep}}$, we follow the proof of \Cref{thm:fd:Q} by using induction,  \Cref{cor:preserve.H.cond.with.HIT} and the Schwartz--Zippel bounds.
If $f_d$ is $2$-irreducible and it has an absolutely irreducible factor which cannot be defined over $K^{\mathrm{sep}}$, the induction argument stays the same, where we now  use \Cref{thm:eff.Hil.lemma} to preserve $2$-irreducibility (together with  \Cref{prop:preservation.H.cond} to preserve NCC), instead of using \Cref{cor:preserve.H.cond.with.HIT} which is based on \Cref{thm:HIT.v2}.
We note that in the latter case (with $d \leq \kappa$) we may assume $\log H_K(f)\ll_{K,n}\log B$, since otherwise we may use arguments similar to
\cite[Theorem 4]{Heath-Brown-Ann} to get $N_{\rm aff}(f, B)\leq O(1) d^2 B^{n-2}=O(1)B^{n-2}$.
\end{proof}

\begin{thm}[Uniform dimension growth for projective varieties]\label{thm:dgcdegree:proj:K}
Let $K$ be a global field. Let $e = 7$ if $\charac K = 0$ and $e = 11$ if $\charac K > 0$. Given $n\geq 3$ an integer, there exist constants $c = c(n,K)$ and $\kappa=\kappa(n,K)$ such that for any integral hypersurface $X\subset \PP_K^n$, and for all $B\geq 2$ we have
\begin{align*}
N(X, B)&\leq cd^e B^{d_K(n-1)}, &\text{ if } d\geq 5, \\
N(X, B)&\leq c B^{d_K(n-1)} (\log B)^\kappa, &\text{ if } d = 4, \\
N(X, B)&\leq c B^{d_K(n-2 + 2/\sqrt{3})} (\log B)^\kappa, &\text{ if } d = 3.
\end{align*}
\end{thm}
\begin{proof}
Follows from \Cref{thm:fd:K}.
\end{proof}

Inspired by \cite{Salb.upcoming} and an email correspondence with Per Salberger, we now give a further variant to Theorem \ref{thm:fd:K}, where we change the conditions on $f_d$ to being $2$-irreducible  over $K$ (resp.~$1$-irreducible).

\begin{thm}\label{thm:final}
Let $K$ be a global field with ring of integers $\cO_K$. Let $e = 7$ if $\charac K = 0$ and $e = 11$ if $\charac K > 0$. Given $n\geq 3$ an integer, there exist constants $c = c(n)$ and $\kappa= \kappa(n)$ such that for all polynomials $f$ in $\cO_K[y_1, \ldots, y_n]$ of degree $d\geq 3$ such that $f$ is irreducible over $K$, and, whose homogeneous degree $d$ part $f_d$ contains at least one absolutely irreducible factor of degree $>2$, one has for all $B\geq 2$ that
\begin{align*}
N_{\rm aff}(f, B)&\leq cd^e B^{n-2}, &\text{ if } d\geq 5, \\
N_{\rm aff}(f, B)&\leq c B^{n-2} (\log B)^\kappa, &\text{ if } d = 4, \\
N_{\rm aff}(f, B)&\leq c B^{n-3 + 2/\sqrt{3}} (\log B)^\kappa, &\text{ if } d = 3.
\end{align*}
If $f_d$ only has at least one absolutely irreducible factor degree $>1$, then
\begin{align*}
N_{\rm aff}(f, B)&\leq cd^e B^{n-2} (\log B)^\kappa, &\text{ if } d \geq 4, \\
N_{\rm aff}(f, B)&\leq c B^{n-3 + 2/\sqrt{3}} (\log B)^\kappa, &\text{ if } d = 3.
\end{align*}
\end{thm}

\begin{proof}
We follow the proof of \Cref{thm:fd:K}, where the key point is that \Cref{prop:count.on.n-2.planes}, about counting on $(n-2)$-planes, still holds. 
Note that since $f_d$ has at least one non-linear absolutely irreducible factor, it is NCC. 
We now prove this proposition under these weaker assumptions on $f_d$.

As usual, we may assume that $f$ is absolutely irreducible. Let $X\subset \AA^n_K$ be the affine variety $V(f)$, and let $X_\infty\subset \PP^{n-1}_K$ be the intersection of $X$ with the hyperplane at infinity.
 By assumption there exists a  geometrically irreducible component $Y$ of $X_\infty$ of degree at least $3$. If $f_d$ contains an irreducible factor of degree $1$ or $2$, then it also has an absolutely irreducible factor of degree $1$ or $2$, let $Z\subset X_\infty$ be the corresponding geometrically irreducible component.
We claim that there are only finitely many $(n-2)$-planes on $X$ which intersect $Z$. Suppose towards a contradiction that this is not true. Since $X$ is geometrically irreducible, $X$ is the union of the infinitely many $(n-2)$-planes contained in it which intersect $Z$.
In particular, by considering the closure of $X$ in $\PP^n_K$ this holds for the points of $Y$. So through every point of $Y$ there is an $(n-2)$-plane contained in $X$ intersecting $Z$. But then because $\deg Y\geq 3$ and $Y\neq Z$, it would follow that the entire plane at infinity is contained in $X$, contradiction. Now, in the proof of \Cref{prop:count.on.n-2.planes} we no longer have to consider rulings of $X$ coming from components of $X_\infty$ of degree $1$ or $2$. We can then simply follow the proof of \Cref{prop:count.on.n-2.planes} in the same way to conclude.

If $f_d$ has at least one absolutely irreducible factor of degree at least $2$, the above proof goes through to show that we can disregard linear components of $X_\infty$.
\end{proof}

Note that Conjecture \ref{conj:uniform.curve.conjecture} and Theorem \ref{thm:fd:Q:conj} can be generalized correspondingly, namely, with the $1$-irreducibility being changed to $f_d$
having at least one absolutely irreducible factor of degree $>1$, while assuming $f$ to be irreducible over $K$. We leave their generalizations to a global field $K$ instead of $\QQ$ to the reader.

\bibliographystyle{amsalpha}
\bibliography{anbib}

\end{document}